\documentclass[reqno]{amsart}
\usepackage{amsfonts}
\usepackage{a4wide}
\usepackage{cite}
\usepackage[linktocpage]{hyperref}
\usepackage{cleveref}

\usepackage{amsmath,amssymb,amsthm,amsfonts}
\usepackage{mathrsfs}
\usepackage{bbm}

\usepackage{color}

\newtheorem{lemma}{Lemma}[section]
\newtheorem{theorem}{Theorem}[section]

\numberwithin{equation}{section}

\newcommand{\dis}{\displaystyle}

\newcommand{\R}{\mathbb{R}}

\renewcommand{\S}{\mathbb{S}}
\newcommand{\T}{\mathbb{T}}

\newcommand{\CE}{\mathcal{E}}

\newcommand{\para}{\shortparallel}

\newcommand{\na}{\nabla}

\newcommand{\al}{\alpha}
\newcommand{\be}{\beta}
\newcommand{\ga}{\gamma}

\newcommand{\la}{\lambda}

\newcommand{\pa}{\partial}

\newcommand{\eps}{\epsilon}

\newcommand{\Ga}{\Gamma}

\begin{document}
\title[Boltzmann equation for soft potentials]{Large amplitude solutions in $L^p_vL^\infty_TL^\infty_x$ to the Boltzmann equation for soft potentials}
\author[Z.~G.~Li]{Zongguang LI}

\begin{abstract}
In this paper we consider the Cauchy problem on the angular cutoff Boltzmann equation near global Maxwillians for soft potentials either in the whole space or in the torus. We establish the existence of global unique mild solutions in the space $L^p_vL^{\infty}_{T}L^{\infty}_{x}$ with polynomial velocity weights for suitably large $p\leq \infty$, whenever for the initial perturbation the weighted $L^p_vL^{\infty}_x$ norm can be arbitrarily large but the $L^1_xL^\infty_v$ norm and the defect mass, energy and entropy are sufficiently small. The proof is based on the local in time existence as well as the uniform a priori estimates via an interplay in $L^p_vL^{\infty}_{T}L^{\infty}_{x}$ and $L^{\infty}_{T}L^{\infty}_{x}L^1_v$.
\end{abstract}
\maketitle
\thispagestyle{empty}

\tableofcontents

\section{Introduction}
We are concerned with the Cauchy problem on the Boltzmann equation
\begin{eqnarray}\label{BE}
&\dis \pa_tF+v\cdot \na_x F=Q(F,F),   \quad &\dis F(0,x,v)=F_0(x,v),
\end{eqnarray}
where $F(t,x,v)\geq0$ is the density distribution function of gas particles with position $x\in \Omega=\R^3$ or $\T^3$ and velocity $v\in \R^3$ at time $t\geq 0$. The bilinear collision operator $Q$ acting only on velocity variable is given by
\begin{equation*}
Q(G,F)(v)=\int_{\R^3}\int_{\S^2}B(v-u,\theta)\left[ G(u')F(v')-G(u)F(v)\right]d\omega du.
\end{equation*}
In this paper, we consider soft potentials under the Grad's angular cutoff assumption. Thus, the collision kernel $B(v-u,\theta)$ takes the form of
\begin{equation}\label{ass.ck}
B(v-u,\theta)=|v-u|^\gamma b(\theta),
\end{equation}
where $-3<\ga<0$ and $0\leq b(\theta) \leq C|\cos \theta|$ for some positive constant $C$ with $\cos \theta=\frac{(v-u)\cdot \omega}{|v-u|}$. The post-collision velocities $v'$ and $u'$ satisfy
\begin{align}\label{velocity}
\begin{split}
v'=v-\left[(v-u)\cdot \omega \right]\omega, \quad &u'=u+\left[(v-u)\cdot \omega \right]\omega,\\
u'+v'=u+v,\quad |u'|^2&+|v'|^2=|u|^2+|v|^2.
\end{split}
\end{align}

Let the global Maxwillian $\mu$ be denoted by
$$
\mu(v)=\frac{1}{(2\pi)^\frac{3}{2}}\exp \left( -\frac{|v|^2}{2} \right).
$$
Moreover, we assume that the following conservation laws and the entropy inequality hold for any solution $F(t,x,v)$ to \eqref{BE} respectively:
\begin{align} 
M_0:=\int_{\Omega}\int_{\R^3} \{F(t,x,v)-\mu(v) \}dvdx&=\int_{\Omega}\int_{\R^3} \{F_0(x,v)-\mu(v) \}dvdx, \label{M}\\
J_0:=\int_{\Omega}\int_{\R^3} v\{F(t,x,v)-\mu(v) \}dvdx&=\int_{\Omega}\int_{\R^3} v\{F_0(x,v)-\mu(v) \}dvdx,\notag
\\
E_0:=\int_{\Omega}\int_{\R^3} |v|^2\{F(t,x,v)-\mu(v) \}dvdx&=\int_{\Omega}\int_{\R^3} |v|^2\{F_0(x,v)-\mu(v) \}dvdx,\label{E}
\end{align}
and 
\begin{align}
\int_{\Omega}\int_{\R^3} \{ F(t,x,v)\log{F(t,x,v)}&-\mu(v)\log{\mu(v)}\}dvdx\notag\\
&\leq \int_{\Omega}\int_{\R^3} \left\{F_0(x,v)\log{F_0}(x,v)-\mu(v)\log{\mu(v)}\right\}dvdx.
\label{H}
\end{align}
For given initial data $F_0(x,v)$ we call $M_0$, $J_0$, $E_0$ and $\iint (F_0\ln F_0-\mu \ln \mu)$ by the defect mass, momentum, energy and entropy, respectively. Using the similar notations as \cite{DHWY}, we define
$$
\CE(F(t)):=\int_{\Omega}\int_{\R^3} \left\{ F(t,x,v)\log{F(t,x,v)}-\mu(v)\log{\mu(v)}\right\}dvdx+(\frac{3}{2}\log (2\pi)-1)M_0+\frac{1}{2}E_0,  
$$
with the initial datum
$\CE(F_0):=\CE(F(0))$. Note that it can be verified that $\CE(F(t))\geq 0$ for any $t\geq 0$, in particular, $\CE(F_0)\geq 0$.

The Boltzmann equation, which is a fundamental mathematical model in collisional kinetic theory, describes the behavior of rarefied gas in non-equilibrium state. There are extensive literatures for the initial and/or boundary value problems of the Boltzmann equation, e.g. \cite{CIP,Villani} and the references therein. The well-known global existence result of renormalized solutions for general $L^1_{x,v}$ initial data with finite mass, energy and entropy was proved by DiPerna-Lions \cite{DL} where the uniqueness of such solutions remains unknown. In the perturbation framework near global Maxwellians, Grad \cite{Grad} studied the linearized operator and Ukai \cite{Ukai} developed the spatially inhomogeneous well-posedness theory by the spectral analysis and the bootstrap argument, see also \cite{NI,Shizuta,UY}. For the enormous works of the linearized operator, interested readers may also refer to Ellis-Pinsky \cite{EP}, Baranger-Mouhot \cite{BM} and the references therein. The energy method in Sobolev spaces was developed through the macro-micro decomposition by Liu-Yang-Yu \cite{LYY} and Guo \cite{Guo04}.

In contrast with the hard potentials, the collision frequency $\nu(v)\sim(1+|v|)^\ga$ in case of soft potentials  $-3<\ga<0$ has no strictly positive lower bound and we are lack of the spectral gap of the linearized operator. For $-1<\ga<0$, based on the decay in time for the linearized equation and the bootstrap argument on the nonlinear equation, Caflisch \cite{Caflisch1,Caflisch2} studied the global existence and large-time behavior of the solutions in $\T^3$. In $\R^3$, the global solution and large-time behavior were solved through the semi-group theory, which was established by Ukai-Asano \cite{UA}. When $-3<\ga<0$, Guo \cite{Guo03} constructed the global classical solutions and Guo-Strain \cite{SG1,SG} proved the large-time behavior. 

Among the works in perturbation framework mentioned above, the initial data should have small oscillations near the global Maxwellian. In the large amplitude situation, Duan-Huang-Wang-Yang \cite{DHWY} developed an $L^\infty_x L^1_v \cap L^\infty_{x,v}$ approach to obtain the global existence and uniqueness of mild solutions in $\R^3$ or $\T^3$ for $-3<\ga\leq1$ in the condition that both $\CE(F_0)$ and the $L^1_x L^\infty_v$ norm of $(F_0-\mu)/\sqrt{\mu}$ are small enough, while the $L^\infty_{x,v}$ norm of $\langle v\rangle^\beta(F_0-\mu)/\sqrt{\mu}$ is only required to be bounded for suitably large $\be$. The smallness in $L^\infty_{x,v}$ is replaced by the smallness in $L^1_xL^\infty_v$ so that the initial data is allowed to have large amplitude around the global Maxwellian with respect to space variable. Motivated by \cite{DHWY} and \cite{Guo}, Nishimura \cite{Nishimura} obtained the global existence for hard potentials in $L^p_vL^\infty_TL^\infty_x$ for large $p$ in order to reduce $L^\infty_v$ to $L^p_v$ with finite $p$. However, the well-posedness theory in such spaces for soft potentials seems still left open.

Now we prepare to state the main results of this paper. Since we need to consider the solutions around the global Maxwillian, we define the perturbation function 
$$
f(t,x,v)=\frac{F(t,x,v)-\mu(v)}{\sqrt{\mu(v)}}.
$$ 
Substituting it into \eqref{BE}, we obtain a Cauchy problem for $f(t,x,v)$ of the form
\begin{eqnarray}\label{PBE}
&\dis \pa_tf+v\cdot \na_x f+\nu(v)f-Kf=\Ga(f,f),   \quad &\dis f(0,x,v)=f_0(x,v),
\end{eqnarray}
where the collision frequency $\nu(v)$, the operator $K$ and the nonlinear term $\Ga$ are respectively given by
$$
\nu(v)=\int_{\R^3}\int_{\S^2}B(v-u,\theta)\mu(u)d\omega du\sim (1+|v|)^\ga,
$$
$$
(Kf)(v)=\int_{\R^3}\int_{\S^2}B(v-u,\theta)\sqrt{\mu(u)}\left( \sqrt{\mu(u')}f(v')+\sqrt{\mu(v')}f(u')-\sqrt{\mu(v)}f(u) \right)d\omega du.
$$
$$
\Ga(f,f)=\Ga_+(f,f)-\Ga_-(f,f), \quad \Ga_\pm(f,f)=\frac{1}{\sqrt{\mu}}Q_\pm({\sqrt{\mu}f},{\sqrt{\mu}f}),
$$
with
$$
Q_+(f,g)=\int_{\R^3}\int_{\S^2}B(v-u,\theta)f(v')g(u')d\omega du,\quad
Q_-(f,g)=\int_{\R^3}\int_{\S^2}B(v-u,\theta)f(v)g(u)d\omega du.
$$
The velocity weight function is denoted by $w_\beta(v)=(1+|v|^2)^{\frac{\be}{2}}\sim (1+|v|)^\be$. Since our results and proofs do not rely on the derivatives of the weighted function, both forms of $w_\beta(v)$ are equivalent. Then from \eqref{PBE}, by integrating along the backward trajectory, we obtain the mild form
\begin{align}\label{mild}
\dis f(t,x,v)=&e^{-\nu(v)t}f_0(x-vt,v)+\int_0^t e^{-\nu(v)(t-s)}(Kf)(s,x-v(t-s),v)ds \notag\\
&+\int_0^t e^{-\nu(v)(t-s)}\Ga(f,f)(s,x-v(t-s),v)ds.
\end{align}
Given two funtions $f=f(t,x,v)$ and $f_0=f_0(x,v)$, for any $0\leq T_0\leq T$, the $L^p_vL^{\infty}_{T_0,T}L^{\infty}_{x}$ norm, $L^p_vL^{\infty}_x$ norm and $L_x^1L^\infty_v$ norm are respectively defined by
\begin{align*}
&\|f\|_{L^p_vL^{\infty}_{T_0,T}L^{\infty}_{x}}:=\left\{ \int_{\R^3} \left[\sup_{t\in [T_0,T]}  \sup_{x \in \Omega}|f(t,x,v)|  \right]^p dv\right\}^{\frac{1}{p}},\\
&\|f_0\|_{L^p_vL^{\infty}_x}:=\left\{ \int_{\R^3}\sup_{x \in \Omega}|f_0(x,v)|^p dv\right\}^{\frac{1}{p}},\\
&\|f_0\|_{L_x^1L^\infty_v}:= \int_{\Omega} \left(\sup_{v \in \R^3}|f_0(x,v)| \right) dx .
\end{align*}
If $T_0=0$, we write $\|f\|_{L^p_vL^{\infty}_{T}L^{\infty}_{x}}$ instead of $\|f\|_{L^p_vL^{\infty}_{0,T}L^{\infty}_{x}}$. In this paper, we consider solutions in $L^p_vL^{\infty}_{T}L^{\infty}_{x}$. In the following sections, we will prove the local existence for bounded $L^p_vL^{\infty}_x$ initial data and establish the $L^p_vL^{\infty}_{T}L^{\infty}_{x}\cap L^{\infty}_{T}L^{\infty}_{x}L^1_v$ estimates in order to extend the obtained local solution to a global solution for small $L_x^1L^\infty_v$ initial data with small $\CE(F_0)$.

Throughout the paper, if a constant $C$ depends on some parameters $\be_1,\be_2\cdots$, then we denote it by $C(\be_1,\be_2,\cdots)$ to emphasize the explicit dependence. The main two results of the paper are stated below.

\begin{theorem}[Local existence]\label{local}
Assume \eqref{ass.ck} with $-3<\gamma<0$. Let $p> \max \{6/(5+\ga),\, 4/(3-\ga),\, 3/(3+\ga),\, (2-\ga)/2\}$ and $\be > \max\{3/p',\, 36,\, 6-2\ga\}$, where $\frac{1}{p}+\frac{1}{p'}=1$. Assume $F_0(x,v):=\mu+\sqrt{\mu}f_0\geq 0$ with $\|w_\be f_0\|_{L^p_vL^{\infty}_x}<\infty$. Then there exists a  constant $C_1=C_1(\be,\ga)>0$ and a positive time 
\begin{align}\label{T_1}
T_1:=\frac{1}{6C_1(1+\|w_\be f_0\|_{L^p_vL^{\infty}_x})}>0,
\end{align}
such that the Cauchy problem on the Boltzmann equation \eqref{BE} admits a unique mild solution $F(t,x,v)=\mu+\sqrt{\mu} f(t,x,v)\geq 0$, $(t,x,v)\in[0,T_1]\times \Omega \times \R^3$, in the sense of \eqref{mild}, satisfying
\begin{align}\label{LE}
\left\|w_\be f 
\right\|_{{L^p_vL^{\infty}_{T_1}L^{\infty}_{x}}}\leq 2\left\|w_\be f_0 
\right\|_{L^p_vL^{\infty}_x}.
\end{align}
\end{theorem}

\begin{theorem}[Global existence]\label{global}
Let all the assumptions in Theorem \ref{local} be satisfied. There is a constant $C_2=C_2(\ga,\be)>0$ such that for any constant $M\geq 1$ that can be arbitrarily large, there exists a constant $\eps=\eps(\ga,\be,M)>0$ such that if it holds that $\|w_\be f_0\|_{L^p_vL^{\infty}_x}\leq M$ and
$$
\max\{\CE(F_0),\,\|f_0\|_{L^1_xL^\infty_v}\}\leq \eps,
$$
then the Cauchy problem on the Boltzmann equation \eqref{BE} admits a unique global mild solution
$F(t,x,v)=\mu+\sqrt{\mu} f(t,x,v)\geq 0$, $(t,x,v)\in[0,\infty)\times \Omega \times \R^3$, in the sense of \eqref{mild}, satisfying
\begin{align}\label{GE}
\left\|w_\be f
\right\|_{{L^p_vL^{\infty}_{T}L^{\infty}_{x}}}\leq C_2M^2,
\end{align}
for any $T\geq 0$.
\end{theorem}

The proof of Theorem \ref{local} is based on the fixed point theorem. We first construct an approximation sequence using the perturbed equation. Then we prove that it is a Cauchy sequence in ${L^p_vL^{\infty}_{T}L^{\infty}_{x}}$ provided $p$ is large enough and $T$ is small enough. The difficulty is due to the nonlinear term $\Ga(f^n,f^n)$. We need to prove the norm of $\int^t_0 \left[w_\be \Ga(f^n,f^n)\right](s,x_1,v)ds$ is bounded by $CT\left\|w_\be f^n\right\|^2_{{L^p_vL^{\infty}_{T}L^{\infty}_{x}}}$. When $p=\infty$ as in \cite{DHWY}, we can directly obtain $\left\|w_\be f\right\|^2_{L^{\infty}}$ from $\Ga(f^n,f^n)$ and the rest of the integral can be bounded by $CT$. However, when we consider $L^p_v$ instead of $L^{\infty}_v$ for some $p\in \R$, it is not straightforward to obtain  $\left\|w_\be f\right\|^2_{{L^p_vL^{\infty}_{T}L^{\infty}_{x}}}$ from the point-wise estimate of the nonlinear term. Moreover, the gain term contains $u'$ and $v'$ as variables and the whole integral is taken with respect to $v$. In this paper, we use the transformation $z_{\shortparallel}=(u-v)\cdot \omega$, $z_{\perp}=z-z_{\para}$ as well as multiple integral inequalities to get the ${L^p_vL^{\infty}_{T}L^{\infty}_{x}}$ norm of $w_\be f$ from the nonlinear term.
At last we can obtain the estimates
\begin{align*}
\dis \left\|w_\be f^{n+1}\right\|_{L^p_vL^{\infty}_{T}L^{\infty}_{x}}\leq 2\|w_\be f_0\|_{L^p_vL^{\infty}_x}
\end{align*}
and
\begin{align*}
\dis \left\|w_\be f^{n+2}-w_\be f^{n+1}\right\|_{L^p_vL^{\infty}_{T}L^{\infty}_{x}}\leq \frac{1}{2}\left\|w_\be f^{n+1}-w_\be f^{n}\right\|_{L^p_vL^{\infty}_{T}L^{\infty}_{x}}.
\end{align*}
Then the approximation sequence is a Cauchy sequence. After taking the limit, we yield a unique local solution which is bounded by the initial data. 

Next we sketch the proof of Theorem \ref{global}. To establish the global $L^\infty$  bound, in the previous works such as \cite{GuoY,Guo,Kim,Strain,UY}, the following inequality is applied to estimate $\Ga(f,f)$
\begin{align}\label{PreIneq}
\left| \left[w_\be \Ga(f,f)\right](t,x,v) \right|\leq C\nu(v)\left\| w_\be f(t) \right\|^2_{L^\infty}.
\end{align}
We can infer from the above inequality that the $L^\infty$ smallness of the initial data is necessary. In order to deal with large initial data, as in \cite{DHWY}, we can improve the inequality \eqref{PreIneq} to be
\begin{align}\label{NexIneq}
\left| \left[w_\be \Ga(f,f)\right](t,x,v) \right|\leq C\nu(v)\left\| w_\be f(t) \right\|^\tau_{L^\infty}\left(\int_{\R^3}|f(t,x,v)|dv\right)^{2-\tau},
\end{align}
for some $0\leq \tau \leq 1$. Then due to the hyperbolicity of the Boltzmann equation, one can prove that if $\CE(F_0)$ and $\|f_0\|_{L^1_x L^{\infty}_v}$ are small enough, $\int_{\R^3}|f(t,x,v)|dv$  will be small uniformly in $x$ for $t\geq T_1$, where $T_1$ is a positive number. Then we can obtain the estimate in $L^\infty$ without assuming the initial data to be small. For hard potentials in $L^p_v$ spaces, a similar idea as \eqref{NexIneq} is established in \cite{Nishimura}, which can be applied to yield global solutions. 

For soft potentials, it is difficult to have a good decay property for the operator $K$ after taking integration in $v$. We will introduce a cut-off function as in \cite{SG} to avoid this inconvenience. Also, the point-wise inequality $e^{-\frac{|u|^2}{4}}|v-u|^\ga\leq C(1+|v|)^\ga$ in \cite{Nishimura} does not hold anymore. We need to use various integral inequalities and transformations to control the nonlinear term. Moreover,  there are terms like $\int_0^t e^{-\nu(v)(t-s)}(w_\be \Ga)(f,f)(s,x-v(t-s),v)ds$ which will cause troubles for our analysis,  since it is hard to get $\|w_\be f\|_{L^p_v}$ from those terms if we take the $L^p_v$ norm. Then we point out that the order for taking $L^p_v$ norm and $L^\infty_T$ norm will matter. If we take $L^\infty_T$ first, we can escape from the difficulty stated above. In this way, we establish the inequality
\begin{align*}
\left\|w_{\be-\ga}\Ga_\pm(f,f) \right\|_{L^p_vL^{\infty}_{T_0,T}L^{\infty}_{x}}\leq& C \left\|f\right\|^{a_\pm}_{L^{\infty}_{T_0,T}L^{\infty}_{x}L^1_v} \left\| w_\be f\right\|^{2-a_\pm}_{L^p_vL^{\infty}_{T_0,T}L^{\infty}_{x}},
\end{align*}
for some $0\leq a_\pm \leq1$. Then we will show that $\left\|f\right\|_{L^{\infty}_{T_0,T}L^{\infty}_{x}L^1_v}$ is small under the smallness condition of $\CE(F_0)$ and $\|f_0\|_{L^1_x L^{\infty}_v}$. Finally, \eqref{GE} follows since we can close our $a$ $priori$ assumption.

As for the organization of the paper, in Section 2, we will give some useful properties of the operator $K$ and introduce some notations. In Section 3, we prove theorem \ref{local} which is the local solution result. In Section 4, we deduce the ${L^p_vL^{\infty}_{T}L^{\infty}_{x}}\cap L^{\infty}_{T}L^{\infty}_{x}L^1_v$ estimate and use it to prove Theorem \ref{global}.

\section{Preliminaries}
We will need the following properties of the operator $K$. Details of the proof can be found in \cite{BPT,Glassey}.

\begin{lemma}\label{K}
For $-3< \ga < 0$,  $\left(Kf\right)(v)$ can be written as
\begin{align*}
(Kf)(v)=\int_{\R^3}k(v,\eta)f(\eta)d\eta,
\end{align*}
with 
\begin{align*}
\dis |k(v,\eta)|\leq C|v-\eta|^\ga e^{-\frac{|v|^2}{4}}e^{-\frac{|\eta|^2}{4}}+\frac{C(\ga)}{|v-\eta|^\frac{3-\ga}{2}}e^{-\frac{|v-\eta|^2}{8}}e^{-\frac{\left| |v|^2-|\eta|^2\right|^2}{8|v-\eta|^2}},
\end{align*}
where $C(\ga)$ is a constant depending only on $\ga$. For $\be\in \R$, we have the estimate
\begin{align}\label{Prok}
\int_{\R^3}\left| k(v,\eta)\cdot \frac{w_\be(v)}{w_\be(\eta)} \right|d\eta \leq C(\ga) (1+|v|)^{-1}.
\end{align}
The above inequality still holds after replacing $k(v,\eta)$ by $k(\eta,v)$ since $k(v,\eta)=k(\eta,v)$. 
\end{lemma}

In order to yield the global existence, it is necessary to get more decay in $|v|$ from $K$. We introduce a smooth cut-off function $\chi_m=\chi_m(\tau)$ as in \cite{SG} with $0\leq m\leq 1$, $0\leq \chi_m \leq 1$. Let $\chi_m(\tau)=1$ for $\tau \leq m$ and $\chi_m(\tau)=0$ for $\tau \geq 2m$. Then $K$ can be split into $K=K^m+K^c$ where

\begin{align}\label{Km}
(K^mf)(v)=&\int_{\R^3}\int_{\S^2}B(v-u,\theta)\chi_m(|v-u|)\sqrt{\mu(u)}\notag\\
&\left( \sqrt{\mu(u')}f(v')+\sqrt{\mu(v')}f(u') 
-\sqrt{\mu(v)}f(u) \right)d\omega du.
\end{align}
For $K^c=K-K^m$, we have the following lemma, which provides the decay we need. The proof is given in the appendix of \cite{DHWY}.

\begin{lemma}\label{Kc}
Let $-3< \ga < 0$ and $\be \in \R$. There is a function $l(v,\eta)$ such that
\begin{align}\label{RepKc}
(K^cf)(v)=\int_{\R^3}l(v,\eta)f(\eta)d\eta
\end{align}
with
\begin{align}
&\int_{\R^3}\left| l(v,\eta)\cdot \frac{w_\be(v)}{w_\be(\eta)} \right|d\eta \leq C(\ga) m^{\ga-1} \frac{\nu(v)}{(1+|v|)^2},\label{Prol1}\\
&\int_{\R^3}\left| l(v,\eta)\cdot \frac{w_\be(v)}{w_\be(\eta)} \right|e^{-\frac{|\eta|^2}{20}}d\eta \leq Ce^{-\frac{|v|^2}{100}}, \notag\\
&\int_{\R^3}\left| l(v,\eta)\cdot \frac{w_\be(v)}{w_\be(\eta)} \right|e^{\frac{|v-\eta|^2}{20}}d\eta \leq C(\ga) m^{\ga-1} \frac{\nu(v)}{(1+|v|)^2}\label{Prol4}.
\end{align}
Furthermore, $l(v,\eta)$ also has the same properties as $k(v,\eta)$ that
\begin{align}\label{Prol2}
\int_{\R^3}\left| l(v,\eta)\cdot \frac{w_\be(v)}{w_\be(\eta)} \right|d\eta \leq C(\ga) (1+|v|)^{-1},
\end{align}
and
\begin{align}\label{Estionl}
\dis |l(v,\eta)|\leq C|v-\eta|^\ga e^{-\frac{|v|^2}{4}}e^{-\frac{|\eta|^2}{4}}+\frac{C(\ga)}{|v-\eta|^\frac{3-\ga}{2}}e^{-\frac{|v-\eta|^2}{8}}e^{-\frac{\left| |v|^2-|\eta|^2\right|^2}{8|v-\eta|^2}}.
\end{align}
All the inequalities hold after changing $l(v,\eta)$ to $l(\eta,v)$.
\end{lemma}

Moreover, we need the following smallness property for $K^m$ when $0<m\ll 1$.
\begin{lemma}\label{ProKm}
For $-3< \ga<0$, $p > 3/(3+\ga)$ and $\frac{1}{p}+\frac{1}{p'}=1$, we have the following pointwise bound of $K^m$,
\begin{align}
|(K^mf)(v)|\leq Cm^{\ga+\frac{3}{p'}}e^{-\frac{|v|^2}{10}}& \left[\left(  \int_{\R^3}\int_{\S^2}e^{-\frac{|u'|^2}{4}}|f(v')|^p d\omega du  \right)^{\frac{1}{p}}+\left(  \int_{\R^3}\int_{\S^2}e^{-\frac{|v'|^2}{4}}|f(u')|^p d\omega du  \right)^{\frac{1}{p}}\right.\notag\\
&\ +\left. \left(  \int_{\R^3}\int_{\S^2}e^{-\frac{|v|^2}{4}}|f(u)|^p d\omega du  \right)^{\frac{1}{p}}\right], \label{BKm}
\end{align}
where $u'$, $v'$ are given in \eqref{velocity}. The three terms on the right-hand side of \eqref{BKm} are obtained from the corresponding three terms on the right-hand side of \eqref{Km}.
\end{lemma}

\begin{proof} From the definition of $K_m$ \eqref{Km}, it is direct to see that
\begin{align*}
\dis \left|(K^mf)(v)\right|\leq&\int_{\R^3}\int_{\S^2}B(v-u,\theta)\chi_m(|v-u|)\sqrt{\mu(u)}\notag\\
&\left(  \left|\sqrt{\mu(u')}f(v')\right|+ \left|\sqrt{\mu(v')}f(u') \right|
+ \left|\sqrt{\mu(v)}f(u) \right|\right)d\omega du.
\end{align*}
We prove for the first term on the right-hand side above which contains $ \sqrt{\mu(u')}f(v')$. Noticing the fact that $e^{-\frac{|u|^2}{4}}\leq C e^{-\frac{|v|^2}{10}}$ for $|v-u|\leq 2m$, it holds that
\begin{align}
&\int_{\R^3}\int_{\S^2}B(v-u,\theta)\chi_m(|v-u|)\sqrt{\mu(u)\mu(u')} f(v')d\omega du \notag\\
&\leq Ce^{-\frac{|v|^2}{10}}\left( \int_{\R^3}\int_{\S^2}|v-u|^{\ga p'}e^{-\frac{|u'|^2}{4}}\chi_m(|v-u|)d\omega du \right)^{\frac{1}{p'}}\left( \int_{\R^3}\int_{\S^2}e^{-\frac{|u'|^2}{4}}|f(v')|^pd\omega du \right)^{\frac{1}{p}}.\label{lemma1}
\end{align}
We have $\ga p'>-3$ by our assumption that $p > 3/(3+\ga)$, which yields that
\begin{align*}
&\left( \int_{\R^3}\int_{\S^2}|v-u|^{\ga p'}e^{-\frac{|u'|^2}{4}}\chi_m(|v-u|)d\omega du \right)^{\frac{1}{p'}}\notag\\
&\leq C\left( \int_{\R^3}|v-u|^{\ga p'}\chi_m(|v-u|)du \right)^{\frac{1}{p'}}\notag\\
&\leq C\left( \int_{\R^3}|u|^{\ga p'}\chi_m(|u|)du \right)^{\frac{1}{p'}}\notag\\
&\leq Cm^{\ga+\frac{3}{p'}}.
\end{align*}
Then together with \eqref{lemma1}, it follows that
\begin{align*}
&\int_{\R^3}\int_{\S^2}B(v-u,\theta)\chi_m(|v-u|)\sqrt{\mu(u)\mu(u')} f(v')d\omega du \notag\\
&\leq Cm^{\ga+\frac{3}{p'}}e^{-\frac{|v|^2}{10}}\left( \int_{\R^3}\int_{\S^2}e^{-\frac{|u'|^2}{4}}|f(v')|^pd\omega du \right)^{\frac{1}{p}}.
\end{align*}
The second and third terms in the right-hand side of \eqref{BKm} can be estimated similarly.
\end{proof}
The following lemma will be used frequently in Section 4. For the proof, see \cite[Lemma 2.7]{DHWY} and  \cite{Guo}.

\begin{lemma}\label{Taylor}
Let $F(t,x,v)$ satisfy \eqref{M}, \eqref{E} and \eqref{H}, we have
\begin{align*}
&\int_{\Omega}\int_{\R^3}\frac{|F(t,x,v)-\mu(v)|^2}{\mu(v)}\chi_{\{ |F(t,x,v)-\mu(v)|\leq \mu(v) \}} + |F(t,x,v)-\mu(v)|\chi_{\{ |F(t,x,v)-\mu(v)|\geq \mu(v) \}}dv dx\notag\\
&\leq 4\left(\iint \left\{ F_0\log{F_0}-\mu\log{\mu}\right\}dvdx+(\frac{3}{2}\log (2\pi)-1)M_0+\frac{1}{2}E_0\right):=4\CE(F_0)\ 
\end{align*}
\end{lemma}
In order to simplify our calculations, we define some notations. For given funtions $f=f(t,x,v)$, $g=g(x,v)$ and funtion $l(v,\eta)$ which is defined in \eqref{RepKc},
\begin{align}
\|f(t,v) \|_{L^\infty_x}:=\sup_{x\in \Omega}|f(t,x,v)|,\quad  \|f(v) \|_{L^\infty_{T_0,T}L^\infty_{x}}:=&\sup_{t\in [T_0,T]}\sup_{x\in \Omega}|f(t,x,v)| ,\notag\\
\|f\|_{L^\infty_{T_0,T}L^\infty_{x} L^1_v}:=\sup_{t\in [T_0,T]}\sup_{x\in \Omega} \left(\int_{\R^3} |f(t,x,v)|dv \right), & \quad 
 \|g(v) \|_{L^\infty_{x}}:=\sup_{x\in \Omega}|g(x,v)|,\notag\\
\left\|f(t,x)\right\|_{L^1_v}:=\int_{\R^3} |f(t,x,v)|dv,\quad l_{w_\al}(v,\eta)&:=l(v,\eta)\frac{w_\al(v)}{w_\al(\eta)}.\label{notations}
\end{align}
When $T_0=0$, $\|f(v) \|_{L^\infty_{T}L^\infty_{x}}:=\|f(v) \|_{L^\infty_{0,T}L^\infty_{x}}$ and $\|f\|_{L^\infty_{T}L^\infty_{x} L^1_v}:=\|f\|_{L^\infty_{0,T}L^\infty_{x} L^1_v}$.
\section{Local-in-time Existence}
In this section we consider the local existence of \eqref{BE} with bounded $L^p_vL^{\infty}_{x}$ initial data. Firstly, rewrite the perturbed equation \eqref{PBE} as
\begin{eqnarray}\label{RPBE}
&\dis \pa_tf+v\cdot \na_x f+\nu(v)f-\Ga_-(f,f)=Kf+\Ga_+(f,f).
\end{eqnarray}
Recall that 
\begin{align}\label{ReGama-}
\Ga_-(f,f)(t,x,v)=\frac{1}{\sqrt{\mu}}Q_-({\sqrt{\mu}f},{\sqrt{\mu}f})(t,x,v)=\int_{\R^3}\int_{\S^2}B(v-u,\theta)\left(\sqrt{\mu}f \right)(t,x,u)f(t,x,v)d\omega du.
\end{align}
Notice that from \eqref{ReGama-} and the fact that $\nu(v)=\int_{\R^3}\int_{\S^2}B(v-u,\theta)\mu(u)d\omega du$, we have
\begin{align*}
\left[\nu f+\Ga_-(f,f)\right](t,x,v)=f(t,x,v)\int_{\R^3}\int_{\S^2}B(v-u,\theta)\left[ \mu(u)+\left(\sqrt{\mu}f \right)(t,x,u) \right]d\omega du.
\end{align*}
After integrating along the backward trajectory, we can construct our approximation sequence ${\{f_n\}_{n=1}^\infty}$ from \eqref{RPBE} as following,
\begin{align}
\dis f^{n+1}(t,x,v)=&e^{-{\int^t_0g^n(\tau,x-v(t-\tau),v)d\tau}}f_0(x-vt,v) \notag\\
&+\int_0^t e^{-{\int^t_sg^n(\tau,x-v(t-\tau),v)d\tau}}(Kf^n)(s,x-v(t-s),v)ds \notag\\
&+\int_0^t e^{-{\int^t_sg^n(\tau,x-v(t-\tau),v)d\tau}}\Ga_+(f^n,f^n)(s,x-v(t-s),v)ds, \label{appro}
\end{align}
where $g^n(\tau,y,v)=\int_{\R^3}\int_{\S^2}B(v-u,\theta)\left[ \mu(u)+\left(\sqrt{\mu}f^n \right)(\tau,y,u) \right]d\omega du$, $f^{n+1}(0,x,v)=f_0(x,v)$ and $f^0(t,x,v)=0$. If we define $F^n=\mu+\sqrt{\mu}f^n$, we can write down the corresponding equation for $F^n$ that
\begin{align*}
\dis F^{n+1}(t,x,v)=&e^{-{\int^t_0g^n(\tau,x-v(t-\tau),v)d\tau}}F_0(x-vt,v) \notag\\
&+\int_0^t e^{-{\int^t_sg^n(\tau,x-v(t-\tau),v)d\tau}}Q_+(F^n,F^n)(s,x-v(t-s),v)ds,
\end{align*}
with $F^{n+1}(0,x,v)=F_0(x,v)$ and $F^0(t,x,v)=\mu(v)\geq0$. If we assume that $F^n\geq 0$, then $g^n(\tau,y,v)\geq0$ and $Q_+(F^n,F^n)(s,x-v(t-s),v)\geq0$, which yields $F^{n+1}\geq 0$. By induction on $n$, we have $F^n\geq 0$ for $n=1,2,\cdots$. Then it holds that $g^n(\tau,y,v)=\int_{\R^3}\int_{\S^2}B(v-u,\theta)F^n(\tau,y,u) d\omega du\geq0$.

Once we have the approximation sequence, we can prove that it is uniformly bounded and also a Cauchy sequence. Then after taking the limit, we will obtain a local solution. The uniqueness can be deduced similarly as how we prove the approximation sequence is Cauchy sequence.

 For $(t,x,v)\in[0,T]\times \Omega \times \R^3$, the following inequality holds directly from \eqref{appro},
\begin{align}
\dis \left|w_\beta(v) f^{n+1}(t,x,v)\right|\leq &\left|w_\beta(v) f_0(x-vt,v)\right|+\int_0^t \left|w_\beta(v)(Kf^n)(s,x-v(t-s),v)\right|ds\notag \\
&+\int_0^t \left|w_\beta(v)\Ga_+(f^n,f^n)(s,x-v(t-s),v)\right|ds\notag \\
&=\left|w_\beta(v) f_0(x-vt,v)\right|+I_1(t,x,v)+I_2(t,x,v). \label{LocalPoint}
\end{align}
Obviously the $L^p_vL^{\infty}_{T}L^{\infty}_{x}$ bound of $\left|w_\beta(v) f_0(x-vt,v)\right|$ is $ \| w_\beta f_0 \|_{L^p_v L^\infty_x}$, we only need to care about $I_1$ and $I_2$. Since $(Kf)(v)=\int_{\R^3}k(v,\eta)f(\eta)d\eta$ by Lemma \ref{K}, we have
\begin{align}\label{firstK}
\dis I_1(t,x,v)&=\int_0^t \left|w_\beta(v)(Kf^n)(s,x-v(t-s),v)\right|ds\notag \\
&=\int_0^t 
\left|\int_{\R^3}k(v,\eta)w_\beta(v)f^n(s,x-v(t-s),\eta)d\eta
\right|ds\notag \\
&=\int_0^t 
\left|\int_{\R^3}k(v,\eta)\frac{w_\beta(v)}{w_\beta(\eta)}w_\beta(\eta)f^n(s,x-v(t-s),\eta)d\eta
\right|ds\notag \\
&\leq \int_0^t \int_{\R^3}\left|k(v,\eta)\frac{w_\beta(v)}{w_\beta(\eta)}\right|\left\|\left(w_\beta f^n\right)(s,\eta)\right\|_{L^\infty_x}d\eta ds.
\end{align}
By H\"older's inequality,
\begin{align}\label{BoundI1}
I_1(t,x,v)&\leq \int_0^t \left(\int_{\R^3}\left|k(v,\eta)\right| \left|\frac{w_\beta(v)}{w_\beta(\eta)}\right|^{p'}d\eta\right)^{\frac{1}{p'}}    \left(\int_{\R^3}\left|k(v,\eta)\right|\left\|\left(w_\be f^n\right)(s,\eta)\right\|^p_{L^\infty_x}d\eta\right)^{\frac{1}{p}}ds.
\end{align}
 Recalling from \eqref{Prok} that $\int_{\R^3}\left|k(v,\eta)\right| \left|\frac{w_\beta(v)}{w_\beta(\eta)}\right|^{p'}d\eta $ is bounded, we have
\begin{align}
\dis I_1(t,x,v)&\leq C\int_0^t     \left(\int_{\R^3}\left|k(v,\eta)\right|\left\|\left(w_\be f^n\right)(s,\eta)\right\|^p_{L^\infty_x}d\eta\right)^{\frac{1}{p}}ds\notag \\
&\leq CT \left(\int_{\R^3}\left|k(v,\eta)\right|\left\|\left(w_\be f^n\right)(\eta)\right\|^p_{L^\infty_{T}L^\infty_{x}}d\eta\right)^{\frac{1}{p}}.\label{I_1}
\end{align}
After taking $L^p_vL^{\infty}_{T}L^{\infty}_{x}$ norm, it follows from \eqref{I_1} that
\begin{align}\label{I1}
\|I_1 \|_{L^p_vL^{\infty}_{T}L^{\infty}_{x}}&\leq CT \left( \int_{\R^3}\int_{\R^3}\left|k(v,\eta)\right|dv\left\|\left(w_\be f^n\right)(\eta)\right\|^p_{L^\infty_{T}L^\infty_{x}}d\eta\right)^{\frac{1}{p}}\notag \\
&\leq CT  \left( \int_{\R^3}\left\|\left(w_\be f^n\right)(\eta)\right\|^p_{L^\infty_{T}L^\infty_{x}}d\eta\right)^{\frac{1}{p}}\notag \\
&\leq CT \left\|w_\be f^n\right\|_{L^p_vL^{\infty}_{T}L^{\infty}_{x}}.
\end{align}
Next we turn to $I_2(t,x,v)$. Denote $x_1=x-v(t-s)$, we obtain that
\begin{align}\label{EI_2}
\dis I_2(t,x,v)&=\int_0^t \left|w_\beta(v)\Ga_+(f^n,f^n)(s,x-v(t-s),v)\right|ds\notag \\
&= \int_0^t \left| \int_{\R^3} \int_{\S^2} |v-u|^\gamma b(\theta)w_\beta(v)e^{-\frac{|u|^2}{4}}f^n(s,x_1,u')f^n(s,x_1,v')d\omega du \right|ds\notag \\
&\leq CT \int_{\R^3} \int_{\S^2} |v-u|^\gamma|\cos\theta|w_\beta(v)e^{-\frac{|u|^2}{4}}\|f^n(u')f^n(v') \|_{L^\infty_{T}L^\infty_{x}}d\omega du.
\end{align}
 Since $|v|^2\leq |u'|^2+|v'|^2$, either $|v|^2\leq 2|u'|^2$ or $|v|^2\leq 2|v'|^2$. Then there exists a strictly positive constant $C$ such that $w_\be(v)\leq w_\be(v)\chi_{\{ |v|^2\leq 2|u'|^2 \}}+w_\be(v)\chi_{\{ |v|^2\leq 2|v'|^2 \}}\leq C\left(w_\be(u')+w_\be(v')\right)$. By this inequality, \eqref{EI_2} and the fact that we can exchange $u'$ and $v'$ by a rotation, we have
\begin{align}
\dis I_2(t,x,v)&\leq CT \int_{\R^3} \int_{\S^2} |v-u|^\gamma|\cos\theta|\left(w_\beta(u')+w_\beta(v')\right)e^{-\frac{|u|^2}{4}}\|f^n(u')f^n(v') \|_{L^\infty_{T}L^\infty_{x}}d\omega du \notag \\
&\leq CT \int_{\R^3} \int_{\S^2} |v-u|^\gamma|\cos\theta|w_\beta(u')e^{-\frac{|u|^2}{4}}\|f^n(u') \|_{L^\infty_{T}L^\infty_{x}} \|f^n(v') \|_{L^\infty_{T}L^\infty_{x}}d\omega du \notag \\
&\leq CT \int_{\R^3} \int_{\S^2} |v-u|^\gamma|\cos\theta|\frac{e^{-\frac{|u|^2}{4}}}{(1+|v'|)^\be}\|(w_\beta f^n)(u') \|_{L^\infty_{T}L^\infty_{x}}\|(w_\beta f^n)(v') \|_{L^\infty_{T}L^\infty_{x}} d\omega du \notag \\
&\leq CT \left(\int_{\R^3} \int_{\S^2} \left| |v-u|^\gamma|\cos\theta|\frac{1}{(1+|v'|)^\be}e^{-\frac{|u|^2}{4}} \right|^{p'}d\omega du \right)^\frac{1}{p'}\notag \\
& \quad \times \left(\int_{\R^3}\|(w_\beta f^n)(u') \|^p_{L^\infty_{T}L^\infty_{x}} \|(w_\beta f^n)(v') \|^p_{L^\infty_{T}L^\infty_{x}} du\right)^\frac{1}{p}.\label{I_2}
\end{align}
We define
\begin{align*}
\dis \tilde{I_1}:=&\int_{\R^3} \int_{\S^2} \left| |v-u|^\gamma|\cos\theta|\frac{1}{(1+|v'|)^\be}e^{-\frac{|u|^2}{4}} \right|^{p'}d\omega du.
\end{align*}
Then it follows from \eqref{I_2} that \begin{align}\label{I_2S}
\dis I_2(t,x,v)\leq CT \left(\tilde{I_1}\right)^\frac{1}{p'}\left(\int_{\R^3}\|(w_\beta f^n)(u') \|^p_{L^\infty_{T}L^\infty_{x}} \|(w_\beta f^n)(v') \|^p_{L^\infty_{T}L^\infty_{x}} du\right)^\frac{1}{p}.
\end{align}
 Denote $z=u-v$, $z_{\shortparallel}=(u-v)\cdot \omega$, $z_{\perp}=z-z_{\para}$. We assume that $p>\max\{6/(5+\ga),\, 3/(3+\ga), \, 4/(3-\ga)\}$ which implies $\frac{\ga -1}{2}p' > -3$, $\frac{\ga +1}{2}p' -2> -3$ and $\frac{\ga+1}{2}p'-2<0$ respectively. Here $3/(3+\ga)$ can be replaced by $2/(3+\ga)$, but we use $3/(3+\ga)$ because of \eqref{lemma1}. Also we require $\be > 3/p'$, then it holds that
\begin{align}
\dis \tilde{I_1}&\leq \int_{\R^3}\int_{z_\perp}\left( \frac{|z_\para|}{|z|^{1-\ga}} \right)^{p'} \frac{1}{|z_\para|^2}e^{-\frac{|z+v|^2}{4}p'}\frac{1}{\left( 1+|v+z_\para| \right)^{\be p'}}dz_\perp dz_\para \notag\\
&\leq \int_{\R^3}\int_{z_\perp}|z_\perp|^{\frac{\ga-1}{2}p'}e^{-\frac{|z_\perp+y|^2}{4}p'}dz_\perp\,|y-v|^{\frac{\ga+1}{2}p'-2}\frac{1}{\left( 1+|y| \right)^{\be p'}} dy \quad \quad (y=v+z_\para).\label{trans}\end{align}
It follows from our assumption $-3<\frac{\ga -1}{2}p'<0 $ that 
$$
\int_{z_\perp}|z_\perp|^{\frac{\ga-1}{2}p'}e^{-\frac{|z_\perp+y|^2}{4}p'}dz_\perp \leq C(1+|y|)^{\frac{\ga-1}{2}p'}\leq C
$$
for some constant $C$. Thus, substituting the inequality above into \eqref{trans}, we have
\begin{align}
\dis \tilde{I_1}&\leq C \int_{\R^3}|y-v|^{\frac{\ga+1}{2}p'-2}\frac{1}{\left( 1+|y| \right)^{\be p'}}dy\notag\\
&\leq C (1+|v|)^{\frac{\ga+1}{2}p'-2}\leq C. \label{I_22}
\end{align}
The second equality above holds since $\frac{\ga +1}{2}p' -2> -3,\, \be > 3/p'$. For the last inequality in \eqref{I_22}, we use the condition $\frac{\ga+1}{2}p'-2<0$.
By \eqref{I_2S}, \eqref{I_22} and $dudv=du'dv'$, after taking $L^p_vL^{\infty}_{T}L^{\infty}_{x}$ norm, we deduce that
\begin{align}\label{I2}
\|I_2 \|_{L^p_vL^{\infty}_{T}L^{\infty}_{x}}&\leq CT \left(\int_{\R^3}\int_{\R^3}\|(w_\beta f^n)(u') \|^p_{L^\infty_{T}L^\infty_{x}} \|(w_\beta f^n)(v') \|^p_{L^\infty_{T}L^\infty_{x}} dudv\right)^\frac{1}{p}\notag\\
&= CT \left(\int_{\R^3}\int_{\R^3}\|(w_\beta f^n)(u) \|^p_{L^\infty_{T}L^\infty_{x}} \|(w_\beta f^n)(v) \|^p_{L^\infty_{T}L^\infty_{x}} dudv\right)^\frac{1}{p}\notag\\
&\leq CT \left\|w_\be f^n\right\|^2_{L^p_vL^{\infty}_{T}L^{\infty}_{x}}.
\end{align}
According to the obervation above,  we can obtain the upper bound of $w_\beta f^n$. It follows from \eqref{LocalPoint}, \eqref{I1} and \eqref{I2} that
\begin{align}\label{IniEstifn}
\dis  \left\|w_\be f^{n+1}\right\|_{L^p_vL^{\infty}_{T}L^{\infty}_{x}}& \leq\left\|w_\be f_0\right\|_{L^p_vL^{\infty}_{x}}+C_1T\left(  \left\|w_\be f^{n}\right\|_{L^p_vL^{\infty}_{T}L^{\infty}_{x}}+\left\|w_\be f^{n}\right\|^2_{L^p_vL^{\infty}_{T}L^{\infty}_{x}} \right),
\end{align}
for some constant $C_1>1$. We set 
\begin{align}\label{DefT1}
\dis T_1=\frac{1}{6C_1(1+\|w_\be f_0\|_{L^p_vL^{\infty}_x})},
\end{align}
then it holds from \eqref{IniEstifn} and \eqref{DefT1} that
\begin{align}\label{LE2}
\dis \left\|w_\be f^{n+1}\right\|_{L^p_vL^{\infty}_{T}L^{\infty}_{x}}\leq 2\|w_\be f_0\|_{L^p_vL^{\infty}_x}.
\end{align}
With this uniform upper bound, we can prove the approximation sequence is a Cauchy sequence. By taking the difference between $w_\beta f^{n+2}$ and $w_\beta f^{n+1}$ and recalling the definition of $f^n$ \eqref{appro}, it holds that
\begin{align*}
\dis &w_\beta(f^{n+2}-f^{n+1})(t,x,v)\notag\\
&=  w_\beta(v) f_0(x-vt,v)\left( e^{-{\int^t_0g^{n+1}(\tau,x-v(t-\tau),v)d\tau}}-e^{-{\int^t_0g^n(\tau,x-v(t-\tau),v)d\tau}}  \right)\notag\\
&\quad +\int_0^t w_\beta(v)\left(Kf^{n+1}\right)(s,x-v(t-s),v)\left(e^{-{\int^t_sg^{n+1}(\tau,x-v(t-\tau),v)d\tau}}-e^{-{\int^t_sg^n(\tau,x-v(t-\tau),v)d\tau}}  \right) ds\notag\\
&\quad +\int_0^t w_\beta(v)\Ga_+(f^{n+1},f^{n+1})(s,x-v(t-s),v)\notag \\
& \quad\quad \times\left(e^{-{\int^t_sg^{n+1}(\tau,x-v(t-\tau),v)d\tau}}-e^{-{\int^t_sg^n(\tau,x-v(t-\tau),v)d\tau}}\right)ds\notag \\
&\quad +\int_0^t e^{-{\int^t_0g^n(\tau,x-v(t-\tau),v)d\tau}} w_\beta(v)\left(Kf^{n+1}-Kf^{n}\right)(s,x-v(t-s),v) ds \notag\\
&\quad +\int_0^t e^{-{\int^t_0g^n(\tau,x-v(t-\tau),v)d\tau}} w_\beta(v)\left(\Ga_+(f^{n+1},f^{n+1})-\Ga_+(f^{n},f^{n})\right)(s,x-v(t-s),v) ds,
\end{align*}
for $(t,x,v)\in[0,T_1]\times\Omega\times\R^3$. Noticing $g^n\geq 0$ for $n=1,2,\cdots$ and $|e^{-a}-e^{-b}|\leq|a-b|$ for any $a,b\geq0$, we have the following inequality for $s\in[0,t]$, 
$$
\left|e^{-{\int^t_sg^{n+1}(\tau,x-v(t-\tau),v)d\tau}}-e^{-{\int^t_sg^n(\tau,x-v(t-\tau),v)d\tau}}\right|\leq \int^t_s\left| (g^{n+1}-g^n)(\tau,x-v(t-\tau),v) \right|d\tau.
$$
Obviously we also have $\left|e^{-{\int^t_sg^n(\tau,x-v(t-\tau),v)d\tau}}\right|\leq1$. Hence we obtain the pointwise bound
\begin{align}
\dis \left| w_\beta(f^{n+2}-f^{n+1})(t,x,v) \right|\leq \tilde{F_1}(t,x,v)+\tilde{F_2}(t,x,v), \label{diff}
\end{align}
where
\begin{align}\label{deftildeF1}
\dis \tilde{F}_1(t,x,v)&:=\left|w_\beta(v) f_0(x-vt,v)\right|\int^t_0\left| (g^{n+1}-g^n)(\tau,x-v(t-\tau),v) \right|d\tau \notag\\
&\ \ +\int_0^t \left|w_\beta(v)\left(Kf^{n+1}\right)(s,x-v(t-s),v)\right|\int^t_s\left| (g^{n+1}-g^n)(\tau,x-v(t-\tau),v) \right|d\tau ds\notag\\
&\ \ +\int_0^t \left|w_\beta(v)\Ga_+(f^{n+1},f^{n+1})(s,x-v(t-s),v)\right|\int^t_s\left| (g^{n+1}-g^n)(\tau,x-v(t-\tau),v) \right|d\tau ds\notag \\
&\ \ +\int_0^t \left|w_\beta(v)\left(Kf^{n+1}-Kf^{n}\right)(s,x-v(t-s),v)\right| ds,
\end{align}
and
\begin{align}\label{deftildeF2}
\dis
\tilde{F}_2(t,x,v)&:=\int_0^t \left|w_\beta(v)\left(\Ga_+(f^{n+1},f^{n+1})-\Ga_+(f^{n},f^{n})\right)(s,x-v(t-s),v)\right| ds.
\end{align}
Recall that $g^n(\tau,y,v)=\int_{\R^3}\int_{\S^2}B(v-u,\theta)\left[ \mu(u)+\left(\sqrt{\mu}f^n \right)(\tau,y,u) \right]d\omega du$. Since for $p>3/(3+\ga)$, $p'\ga>-3$, by similar arguments as in \eqref{I_2}, one gets that
\begin{align}
\dis
&\int^t_s\left| (g^{n+1}-g^n)(\tau,x-v(t-\tau),v) \right|d\tau\notag \\
& \leq\int^t_0\left| (g^{n+1}-g^n)(\tau,x-v(t-\tau),v) \right|d\tau\notag\\
&\leq CT_1 \int_{\R^3} \int_{\S^2} |v-u|^\gamma|\cos\theta|e^{-\frac{|u|^2}{4}}\|f^{n+1}(u)-f^n(u) \|_{L^\infty_{T_1}L^\infty_{x}}d\omega du \notag \\
&\leq CT_1 \left(\int_{\R^3} \int_{\S^2} \left| |v-u|^\gamma|\cos\theta|e^{-\frac{|u|^2}{4}} \right|^{p'}d\omega du \right)^\frac{1}{p'}  \left(\int_{\R^3}\left\|f^{n+1}(u)-f^n(u)\right\|^p_{L^\infty_{T_1}L^\infty_{x}} du\right)^\frac{1}{p}  \notag \\
&\leq CT_1 \left(\int_{\R^3}\left\|f^{n+1}(u)-f^n(u)\right\|^p_{L^\infty_{T_1}L^\infty_{x}} du\right)^\frac{1}{p}\notag \\
&= CT_1 \left\|f^{n+1}-f^n\right\|_{L^p_vL^\infty_{T_1}L^\infty_{x}}.\label{123}
\end{align}
Also for the last term on the right-hand side of \eqref{deftildeF1}, using similar arguments as in \eqref{firstK}, \eqref{BoundI1} and \eqref{I_1}, we have
\begin{align}
\dis
&\int_0^t \left|w_\beta(v)\left(Kf^{n+1}-Kf^{n}\right)(s,x-v(t-s),v)\right| ds\notag\\
&=\int_0^t \left| \int_{\R^3}k(v,\eta)\frac{w_\beta(v)}{w_\beta(\eta)}\left(w_\be f^{n+1}-w_\be f^n\right)(s,x-v(t-s),\eta)d\eta\right| ds \notag\\
&\leq\int_0^t \left(\int_{\R^3}\left|k(v,\eta)\right| \left|\frac{w_\beta(v)}{w_\beta(\eta)}\right|^{p'}d\eta\right)^{\frac{1}{p'}}    \left(\int_{\R^3}\left|k(v,\eta)\right|\left\|\left(w_\be f^{n+1}-w_\be f^n\right)(s,\eta)\right\|^p_{L^\infty_x}d\eta\right)^{\frac{1}{p}}ds\notag \\
&\leq CT_1 \left(\int_{\R^3}\left|k(v,\eta)\right|\left\|\left(w_\be f^{n+1}-w_\be f^n\right)(\eta)\right\|^p_{L^\infty_{T_1}L^\infty_{x}}d\eta\right)^{\frac{1}{p}}.\label{4}
\end{align}
It follows from \eqref{deftildeF1}, \eqref{123} and \eqref{4} that
\begin{align}\label{first4}
\dis \tilde{F}_1(t,x,v)&\leq CT_1\left\|w_\be f^{n+1}-w_\be f^n\right\|_{L^p_vL^\infty_{T_1}L^\infty_{x}}\notag \\
&\quad\quad\times\left(
\left|w_\beta(v) f_0(x-vt,v)\right|+\int_0^t \left|w_\beta(v)(Kf^{n+1})(s,x-v(t-s),v)\right|ds\right.\notag \\
&\quad\quad\quad\quad\left.+\int_0^t \left|w_\beta(v)\Ga_+(f^{n+1},f^{n+1})(s,x-v(t-s),v)\right|ds
\right).
\end{align}
After taking $L^p_vL^\infty_{T_1}L^\infty_{x}$ norm, by \eqref{LE2}, \eqref{first4} and similar arguments as how we estimate the right-hand side of \eqref{LocalPoint}, we can bound $L^p_vL^\infty_{T_1}L^\infty_{x}$ norm of $\tilde{F_1}(t,x,v)$ as follows:
\begin{align}\label{tildeF}
\dis \|\tilde{F}_1\|_{L^p_vL^\infty_{T_1}L^\infty_{x}}\leq C T_1\left( 1+\left\|w_\be f_0\right\|_{L^p_vL^{\infty}_{x}}\right)\left\|w_\be f^{n+1}-w_\be f^{n}\right\|_{L^p_vL^\infty_{T_1}L^\infty_{x}}.
\end{align}
Next we need to estimate $\tilde{F}_2(t,x,v)$. It is direct to see
\begin{align}
\dis
\tilde{F}_2(t,x,v)&\leq \int_0^t \left|w_\beta(v)\Ga_+(f^{n+1}-f^{n},f^{n})(s,x-v(t-s),v)\right| ds\notag\\
&\quad + \int_0^t \left|w_\beta(v)\Ga_+(f^{n+1},f^{n+1}-f^{n})(s,x-v(t-s),v)\right| ds \label{Ga++}.
\end{align}
We firstly focus on the integral containing $\Ga_+(f^{n+1},f^{n+1}-f^{n})$. By using the similar arguments in \eqref{EI_2} and \eqref{I_2}, we have
\begin{align*}
\dis
&\left|w_\beta(v)\Ga_+(f^{n+1},f^{n+1}-f^{n})(s,x-v(t-s),v)\right| \notag \\
&\leq C\int_{\R^3} \int_{\S^2} |v-u|^\gamma|\cos\theta|w_\beta(v)e^{-\frac{|u|^2}{4}}\|f^{n+1}(u')\left(f^{n+1}-f^{n}\right)(v') \|_{L^\infty_{T_1}L^\infty_{x}}d\omega du  \notag \\
&\leq C\int_{\R^3} \int_{\S^2} |v-u|^\gamma|\cos\theta|e^{-\frac{|u|^2}{4}}\|w_\beta f^{n+1}(u')\left(f^{n+1}-f^{n}\right)(v') \|_{L^\infty_{T_1}L^\infty_{x}}d\omega du  \notag \\
&\quad + C\int_{\R^3} \int_{\S^2} |v-u|^\gamma|\cos\theta|e^{-\frac{|u|^2}{4}}\|f^{n+1}(u')\left(w_\beta f^{n+1}-w_\beta f^{n}\right)(v') \|_{L^\infty_{T_1}L^\infty_{x}}d\omega du  \notag \\
&\leq C\int_{\R^3} \int_{\S^2} |v-u|^\gamma|\cos\theta|e^{-\frac{|u|^2}{4}}\frac{1}{w_\beta(v')} \|w_\beta f^{n+1}(u')\left(w_\beta f^{n+1}-w_\beta f^{n}\right)(v') \|_{L^\infty_{T_1}L^\infty_{x}}d\omega du\notag \\
&\quad + C\int_{\R^3} \int_{\S^2} |v-u|^\gamma|\cos\theta|e^{-\frac{|u|^2}{4}}\frac{1}{w_\beta(u')}\|w_\beta f^{n+1}(u')\left(w_\beta f^{n+1}-w_\beta f^{n}\right)(v') \|_{L^\infty_{T_1}L^\infty_{x}}d\omega du \notag \\
&\leq C\left(\int_{\R^3}\|(w_\be f^n)(u') \|^p_{L^\infty_{T_1}L^\infty_{x}} \|\left(w_\beta f^{n+1}-w_\beta f^{n}\right)(v')\|^p_{L^\infty_{T_1}L^\infty_{x}} du\right)^\frac{1}{p}.
\end{align*}
We can treat $\left|w_\beta(v)\Ga_+(f^{n+1}-f^{n},f^{n})(s,x-v(t-s),v)\right|$ in the same way, then we conclude that 
\begin{align}\label{Ga+++}
\dis
&\int_0^t \left|w_\beta(v)\left(\Ga_+(f^{n+1},f^{n+1})-\Ga_+(f^{n},f^{n})\right)(s,x-v(t-s),v)\right| ds\notag\\
&\leq  CT_1\left(\int_{\R^3}\|(w_\be f^n)(u') \|^p_{L^\infty_{T_1}L^\infty_{x}} \|\left(w_\beta f^{n+1}-w_\beta f^{n}\right)(v') \|^p_{L^\infty_{T_1}L^\infty_{x}} du\right)^\frac{1}{p}.
\end{align}
It follows from \eqref{Ga++} and \eqref{Ga+++} that
\begin{align}\label{tildeF2}
\dis \|\tilde{F}_2\|_{L^p_vL^\infty_{T_1}L^\infty_{x}}\leq C T_1\left\|w_\be f_0\right\|_{L^p_vL^{\infty}_{x}}\,\left\|w_\be f^{n+1}-w_\be f^{n}\right\|_{L^p_vL^\infty_{T_1}L^\infty_{x}},
\end{align}
where $\tilde{F}_2$ is defined in \eqref{deftildeF2}.
Using \eqref{diff}, \eqref{tildeF}, \eqref{tildeF2} and recalling that $T_1=\frac{1}{6C_1(1+\|w_\be f_0\|_{L^p_vL^{\infty}_x})}$ from \eqref{DefT1}, we yield
\begin{align*}
\dis \left\|w_\be f^{n+2}-w_\be f^{n+1}\right\|_{L^p_vL^{\infty}_{T_1}L^{\infty}_{x}}&\leq \|\tilde{F}_1\|_{L^p_vL^\infty_{T_1}L^\infty_{x}}+\|\tilde{F}_2\|_{L^p_vL^\infty_{T_1}L^\infty_{x}}\notag\\
&\leq CT_1\left( 1+\left\|w_\be f_0\right\|_{L^p_vL^{\infty}_{x}}\right)\left\|w_\be f^{n+1}-w_\be f^{n}\right\|_{L^p_vL^{\infty}_{T_1}L^{\infty}_{x}} \notag \\
&\leq \frac{C}{6C_1}\left\|w_\be f^{n+1}-w_\be f^{n}\right\|_{L^p_vL^{\infty}_{T_1}L^{\infty}_{x}}\notag \\
&\leq \frac{1}{2}\left\|w_\be f^{n+1}-w_\be f^{n}\right\|_{L^p_vL^{\infty}_{T_1}L^{\infty}_{x}},
\end{align*}
by choosing $C_1$ large enough such that $\frac{C}{6C_1}\leq \frac{1}{2}$. Then we have proved that the approximation sequence is a Cauchy sequence. After taking the limit, we can see the limit function is a local-in-time solution of \eqref{BE} and satiesfies the conservation laws and entropy inequality. \eqref{LE} follows from letting $n$ tend to infinity in \eqref{LE2}. The uniqueness can be obtained in the same way as how we estimate \eqref{diff}. Up to now, we finish the proof of the local existence.

\section{Global-in-time Existence}
In order to obtain the global existence, we rewrite the mild form \eqref{mild} as
\begin{align}\label{rmild}
\dis f(t,x,v)=&e^{-\nu(v)t}f_0(x-vt,v)+\int_0^t e^{-\nu(v)(t-s)}(K^mf)(s,x-v(t-s),v)ds \notag\\
&+\int_0^t e^{-\nu(v)(t-s)}(K^cf)(s,x-v(t-s),v)ds\notag\\
&+\int_0^t e^{-\nu(v)(t-s)}\Ga(f,f)(s,x-v(t-s),v)ds,
\end{align}
where $K^m$ is defined in \eqref{Km} and $K^c=K-K^m$. 

\subsection{Estimates on $\Ga$}We first introduce the following lemma in order to estimate $\Ga$. 

\begin{lemma}\label{Gama}
Let $\ga$, $\be$ and $p$ satisfy the assumption in Theorem \ref{local} and $3/(3+\ga)< q< p$. Then for any positive $T_0$, $\bar{T}$ with $0\leq T_0\leq \bar{T}$, there is a strictly positive constant $C$ such that 
\begin{align}
\left\|w_{\be-\ga}\Ga_-(f,f) \right\|_{L^p_vL^{\infty}_{T_0,\bar{T}}L^{\infty}_{x}}\leq& C \left\| w_\be f\right\|^{1+\frac{p(q-1)}{q(p-1)}}_{L^p_vL^{\infty}_{T_0,\bar{T}}L^{\infty}_{x}}\left\|f\right\|^{\frac{p-q}{q(p-1)}}_{L^{\infty}_{T_0,\bar{T}}L^{\infty}_{x}L^1_v} \label{Ga-1}\\
\left\|w_{\be-\ga}\Ga_+(f,f) \right\|_{L^p_vL^{\infty}_{T_0,\bar{T}}L^{\infty}_{x}}\leq& C \left\| w_\be f\right\|^{\frac{1}{8}\left(\frac{1}{q}-\frac{1}{p}\right)+1+\frac{r}{p}}_{L^p_vL^{\infty}_{T_0,\bar{T}}L^{\infty}_{x}}\left\|f\right\|^{\frac{1}{8}\left(\frac{1}{q}-\frac{1}{p}\right)}_{L^{\infty}_{T_0,\bar{T}}L^{\infty}_{x}L^1_v} \label{Ga+1},
\end{align}
where $r=p-\frac{p-q}{4q}$
\end{lemma}
\begin{proof}

Assume $T_0\leq t\leq \bar{T}$. We first prove inequality \eqref{Ga-1}. Denote $q'=\frac{q}{q-1}$. By H\"older's inequality, we have
\begin{align}
\dis
&\left|w_{\be-\ga}(v)\Ga_-(f,f)(t,x,v) \right|\notag \\
&= \left| (1+|v|)^{-\ga}\int_{\R^3} \int_{\S^2} |v-u|^\gamma b(\theta)w_\beta(v)e^{-\frac{|u|^2}{4}}f(t,x,u)f(t,x,v)d\omega du  \right| \notag \\
&\leq C (1+|v|)^{-\ga}\left(\int_{\R^3} \int_{\S^2} \left| |v-u|^\gamma|\cos\theta|e^{-\frac{|u|^2}{4}} \right|^{q'}d\omega du \right)^\frac{1}{q'}\notag \\
& \quad \times \left(\int_{\R^3}|f(t,x,u) |^q |(w_\be f)(t,x,v) |^q du\right)^\frac{1}{q}.\label{InequlityGa-}
\end{align}
Notice that we require $q>3/(3+\ga)$, which implies $\ga q'>-3$. Then it holds that
\begin{align*}
&\left(\int_{\R^3} \int_{\S^2} \left| |v-u|^\gamma|\cos\theta|e^{-\frac{|u|^2}{4}} \right|^{q'}d\omega du \right)^\frac{1}{q'}\notag \\
&\leq C\left(\int_{\R^3} \left| |v-u|^\gamma e^{-\frac{|u|^2}{4}} \right|^{q'}du \right)^\frac{1}{q'}\notag \\
&\leq C (1+|v|)^{\ga}.
\end{align*}
We substitute this inequality into \eqref{InequlityGa-} and obtain
\begin{align}
\dis
&\left|w_{\be-\ga}(v)\Ga_-(f,f)(t,x,v) \right|\notag \\
&\leq C|(w_\be f)(t,x,v) |\left(\int_{\R^3} | f(t,x,u) |^q  du\right)^\frac{1}{q}\notag \\
&\leq C|(w_\be f)(t,x,v) |\left(\int_{\R^3}| f(t,x,u) | du\right)^\frac{p-q}{q(p-1)}\left(\int_{\R^3}|(w_\beta f)(t,x,u) |^p du\right)^\frac{(q-1)}{q(p-1)} \label{interpo}
\end{align}
by the interpolation inequality in Lebesgue spaces $\|f\|_{L^q}\leq \|f\|^{\frac{p-q}{q(p-1)}}_{L^1} \|f\|^{\frac{p(q-1)}{q(p-1)}}_{L^p}$ for $1<q<p$. Then the inequality \eqref{Ga-1} follows from \eqref{interpo} by taking the $L^p_vL^{\infty}_{T_0,\bar{T}}L^{\infty}_{x}$ norm.

Next we prove the inequality \eqref{Ga+1}, noticing that we can exchange $u'$ and $v'$ by a rotation and $w_\be(v)\leq C\left(w_\be(v')+w_\be(u') \right)$ for some constant $C$, similar arguments as \eqref{InequlityGa-} yield that
\begin{align*}
\dis
&\left|w_{\be-\ga}(v)\Ga_+(f,f)(t,x,v) \right|\notag \\
&\leq C (1+|v|)^{-\ga}\left(\int_{\R^3} \int_{\S^2}  |v-u|^{\gamma q'}|\cos\theta| ^{\gamma q'}e^{-\frac{|u|^2}{4}}d\omega du \right)^\frac{1}{q'}\notag \\
& \quad \times \left(\int_{\R^3}e^{-\frac{|u|^2}{4}}|(w_\be f)(t,x,u') |^q |f(t,x,v') |^q d\omega du\right)^\frac{1}{q}\notag \\
&\leq C \left(\int_{\R^3} \int_{\S^2}e^{-\frac{|u|^2}{4}}|(w_\be f)(t,x,u') |^q |f(t,x,v') |^q d\omega du \right)^\frac{1}{q}.
\end{align*}
Write $|f(t,x,v') |^q=|f(t,x,v') |^{\frac{p-q}{4p}}|f(t,x,v') |^{q-\frac{p-q}{4p}}$. Applying H\"older's inequality to the last term above, it holds that
\begin{align}
\dis
&\left|w_{\be-\ga}(v)\Ga_+(f,f)(t,x,v) \right|\notag \\
&\leq C  \left(\int_{\R^3} \int_{\S^2}e^{-\frac{|u|^2}{4}}|f(t,x,v') |^\frac{1}{4} d\omega du\right)^{\frac{1}{q}-\frac{1}{p}} \notag \\
& \quad \times
\left(\int_{\R^3} \int_{\S^2}e^{-\frac{|u|^2}{4}}|(w_\be f)(t,x,u') |^p |f(t,x,v') |^r d\omega du\right)^\frac{1}{p},\label{Gama+}
\end{align}
where $r=p-\frac{p-q}{4q}\leq p$. For convenience, we define
\begin{align*}
\dis
\tilde{I_2}:=&\int_{\R^3} \int_{\S^2}e^{-\frac{|u|^2}{4}}|f(t,x,v') |^\frac{1}{4} d\omega du.
\end{align*}
Using tranformation $z=u-v$, $z_{\shortparallel}=(u-v)\cdot \omega$, $z_{\perp}=z-z_{\para}$ as \eqref{trans}, we have
\begin{align}\label{estitildeI2}
\dis
\tilde{I_2}&\leq \int_{\R^3}\int_{z_\perp}e^{-\frac{|z_\perp+\eta|^2}{4}}dz_\perp|f(t,x,\eta)|^\frac{1}{4}\frac{1}{|\eta-v|^{2}} d\eta \notag \\
&\leq \int_{\R^3}|f(s,y,\eta)|^\frac{1}{4}\frac{1}{|\eta-v|^{2}} d\eta
\end{align}
It is direct to get $\be/12>3$ from our assumption that $\be>36$. Then $\int_{\R^3}(1+|\eta|)^{-\frac{\be }{12}}|\eta-v|^{-\frac{8}{3}}d\eta$ will be uniformly bounded in $v$. By $|f(t,x,\eta)|^\frac{1}{4}\leq |f(t,x,\eta)|^\frac{1}{8}|w_{\be/2}f(t,x,\eta)|^\frac{1}{8}(1+|\eta|)^{-\frac{\be }{16}}$, H\"older's inequality and \eqref{estitildeI2}, we obtain
\begin{align}\label{InitildeI2}
\dis \tilde{I_2}&\leq C \int_{\R^3}|f(t,x,\eta)|^\frac{1}{8}|w_{\be/2}f(t,x,\eta)|^\frac{1}{8}\frac{(1+|\eta|)^{-\frac{\be }{16}}}{|\eta-v|^{2}} d\eta \notag \\
&\leq C \left( \int_{\R^3}|f(t,x,\eta)|^\frac{1}{2}|w_{\be/2}f(t,x,\eta)|^\frac{1}{2}d\eta  \right)^\frac{1}{4}\left( \int_{\R^3}\frac{(1+|\eta|)^{-\frac{\be }{12}}}{|\eta-v|^{\frac{8}{3}}}d\eta  \right)^\frac{3}{8}\notag \\
&\leq C\left(\int_{\R^3}| f(t,x,\eta) | d\eta\right)^\frac{1}{8}\left(\int_{\R^3}|(w_{\beta/2} f)(t,x,\eta) | d\eta\right)^\frac{1}{8}.
\end{align}
Using the relation
\begin{align*}
&\left(\int_{\R^3}\left|(w_{\beta/2} f)(t,x,\eta) \right| d\eta\right)\notag \\
&\leq \left(\int_{\R^3}\left|\frac{1}{(1+|v|)^\frac{\be}{2}}\right|^{p'} d\eta\right)^\frac{1}{p'}\left(\int_{\R^3}\left|(w_\beta f)(t,x,\eta)\right|^p d\eta\right)^\frac{1}{p}\notag \\
&\leq C\left(\int_{\R^3}|(w_\beta f)(t,x,\eta) |^p d\eta\right)^\frac{1}{p},
\end{align*}
we have from \eqref{InitildeI2} that
\begin{align*}
\dis
\tilde{I_2}=&\int_{\R^3} \int_{\S^2}e^{-\frac{|u|^2}{4}}|f(t,x,v') |^\frac{1}{4} d\omega du\notag \\
&\leq C\left(\int_{\R^3}| f(t,x,\eta) | d\eta\right)^\frac{1}{8}\left(\int_{\R^3}|(w_\beta f)(t,x,\eta) |^p d\eta\right)^\frac{1}{8p}\notag \\
&\leq C \left\|f\right\|^{\frac{1}{8}}_{L^{\infty}_{T_0,\bar{T}}L^{\infty}_{x}L^1_v} \left\| w_\be f\right\|^{\frac{1}{8}}_{L^p_vL^{\infty}_{T_0,\bar{T}}L^{\infty}_{x}}.
\end{align*}
Together with \eqref{Gama+}, after taking the $L^p_vL^{\infty}_{T_0,\bar{T}}L^{\infty}_{x}$ norm and by $dudv=du'dv'$, \eqref{Ga+1} follows from the fact that
\begin{align*}
\left\|w_{\be-\ga}\Ga_+(f,f) \right\|_{L^p_vL^{\infty}_{T_0,\bar{T}}L^{\infty}_{x}}&\leq C \left\| w_\be f\right\|^{\frac{1}{8}\left(\frac{1}{q}-\frac{1}{p}\right)}_{L^p_vL^{\infty}_{T_0,\bar{T}}L^{\infty}_{x}}\left\|f\right\|^{\frac{1}{8}\left(\frac{1}{q}-\frac{1}{p}\right)}_{L^{\infty}_{T_0,\bar{T}}L^{\infty}_{x}L^1_v}  \notag \\
&\quad  \times
\left(\int_{\R^3} \int_{\R^3}\|(w_\be f)(u') \|_{L^{\infty}_{T_0,\bar{T}}L^{\infty}_{x}}^p \|f(v') \|_{L^{\infty}_{T_0,\bar{T}}L^{\infty}_{x}}^r dudv\right)^\frac{1}{p}\notag \\
&\leq C  \left\| w_\be f\right\|^{\frac{1}{8}\left(\frac{1}{q}-\frac{1}{p}\right)}_{L^p_vL^{\infty}_{T_0,\bar{T}}L^{\infty}_{x}} \left\|f\right\|^{\frac{1}{8}\left(\frac{1}{q}-\frac{1}{p}\right)}_{L^{\infty}_{T_0,\bar{T}}L^{\infty}_{x}L^1_v}\notag \\
&\quad  \times \left(\int_{\R^3}\|f(v) \|_{L^p_vL^{\infty}_{T_0,\bar{T}}L^{\infty}_{x}}^r dv\right)^\frac{1}{p}
\left(\int_{\R^3} \|(w_\be f)(u) \|_{L^p_vL^{\infty}_{T_0,\bar{T}}L^{\infty}_{x}}^p du\right)^{\frac{1}{p}}\notag \\
&\leq C \left\| w_\be f\right\|^{\frac{1}{8}\left(\frac{1}{q}-\frac{1}{p}\right)+1+\frac{r}{p}}_{L^p_vL^{\infty}_{T_0,\bar{T}}L^{\infty}_{x}}\left\|f\right\|^{\frac{1}{8}\left(\frac{1}{q}-\frac{1}{p}\right)}_{L^{\infty}_{T_0,\bar{T}}L^{\infty}_{x}L^1_v} .
\end{align*}
In the last inequality above, we use the inequality
\begin{align*}
\dis &\left(\int_{\R^3}\|f(v) \|_{L^p_vL^{\infty}_{T_0,\bar{T}}L^{\infty}_{x}}^r dv\right)^\frac{1}{p}\notag \\
&\leq \left(\int_{\R^3}\left|\frac{1}{(1+|v|)^{r\be}} \right|^{p'} dv\right)^\frac{1}{p'}
\left(\int_{\R^3}\|w_\be f(v) \|_{L^p_vL^{\infty}_{T_0,\bar{T}}L^{\infty}_{x}}^p dv\right)^{\frac{1}{p}\cdot \frac{r}{p}}\notag \\
&\leq C
\left\| w_\be f\right\|^\frac{r}{p}_{L^p_vL^{\infty}_{T_0,\bar{T}}L^{\infty}_{x}}.
\end{align*}
We have completed the proof of Lemma \ref{Gama}.
\end{proof}

\subsection{Global ${L^p_vL^{\infty}_{T}L^{\infty}_{x}}$ Estimate}Now we can deduce the following result, which allows us to bound the ${L^p_vL^{\infty}_{T}L^{\infty}_{x}}$ norm of $w_\be f$ by the initial data, $\CE(F_0)$ and the product of $\left\|f\right\|_{L^{\infty}_{T}L^{\infty}_{x}L^1_v}$ and $\left\| w_\be f\right\|_{L^p_vL^{\infty}_{T}L^{\infty}_{x}}$.
\begin{lemma}\label{Estimate}
Let all the assumptions in Theorem \ref{local} be satisfied. It holds that
\begin{align}\label{LemmaEstimate}
\| w_\be f\|&_{L^p_vL^{\infty}_{T}L^{\infty}_{x}}\leq  C_2\left\{ \left\| w_\be f_0\right\|_{L^p_vL^{\infty}_{x}}+\left\| w_\be f_0\right\|^2+\sqrt{\CE(F_0)}+\CE(F_0) \right\}\notag\\
&+C_2 \left\| w_\be f\right\|^{1+\frac{p(q-1)}{q(p-1)}}_{L^p_vL^{\infty}_{T}L^{\infty}_{x}} \left\|f\right\|^{\frac{p-q}{q(p-1)}}_{L^{\infty}_{T_1,T}L^{\infty}_{x}L^1_v}+C_2  \left\| w_\be f\right\|^{\frac{1}{8}\left(\frac{1}{q}-\frac{1}{p}\right)+1+\frac{r}{q}}_{L^p_vL^{\infty}_{T}L^{\infty}_{x}} \left\|f\right\|^{\frac{1}{8}\left(\frac{1}{q}-\frac{1}{p}\right)}_{L^{\infty}_{T_1,T}L^{\infty}_{x}L^1_v},
\end{align}
for $T_1$ defined in \eqref{T_1}, $T\geq T_1$ and some constant $C_2>1$. 
\end{lemma}
\begin{proof}
By the mild form \eqref{rmild}, for $(t,x,v)\in[0,T]\times \Omega\times \R^3$, it is noted that
\begin{align}\label{rrmild}
\dis (w_\be f)(t,x,v)&=e^{-\nu(v)t} (w_\be f_0)(x-vt,v)+\int_0^t e^{-\nu(v)(t-s)}(w_\be K^mf)(s,x-v(t-s),v)ds \notag\\
&\quad+\int_0^t e^{-\nu(v)(t-s)}(w_\be K^cf)(s,x-v(t-s),v)ds\notag\\
&\quad+\int_0^t e^{-\nu(v)(t-s)}(w_\be \Ga)(f,f)(s,x-v(t-s),v)ds\notag\\
&=J_0(t,x,v)+J_1(t,x,v)+J_2(t,x,v)+J_3(t,x,v).
\end{align}
We define
\begin{align}\label{tildeJ}
\tilde{J}(v):=&\left(  \int_{\R^3}\int_{\S^2}e^{-\frac{|u'|^2}{4}}\|f(v')\|_{L^{\infty}_{T}L^{\infty}_{x}}^p d\omega du  \right)^{\frac{1}{p}}+\left(  \int_{\R^3}\int_{\S^2}e^{-\frac{|v'|^2}{4}}\|f(u')\|_{L^{\infty}_{T}L^{\infty}_{x}}^p d\omega du \right)^{\frac{1}{p}}\notag\\
&+ \left(  \int_{\R^3}\int_{\S^2}e^{-\frac{|v|^2}{4}}\|f(u)\|_{L^{\infty}_{T}L^{\infty}_{x}}^p d\omega du  \right)^{\frac{1}{p}}.
\end{align}
It follows from Lemma \ref{ProKm} that
\begin{align*}
|J_1(t,x,v)|&\leq \int_0^t e^{-\nu(v)(t-s)}\left|(w_\be K^mf)(s,x-v(t-s),v)\right|ds\notag\\ 
&\leq Cm^{\ga+\frac{3}{p'}}\tilde{J}(v)w_\be(v) e^{-\frac{|v|^2}{10}} \int^t_0e^{-\nu(v)(t-s)}ds  \notag\\
&\leq Cm^{\ga+\frac{3}{p'}}\tilde{J}(v).
\end{align*}
In that last inequality above, we use the fact that $w_\be(v) e^{-\frac{|v|^2}{10}} \int^t_0e^{-\nu(v)(t-s)}ds =\frac{w_\be(v)}{\nu(v)}e^{-\frac{|v|^2}{10}}\leq C$. Then after taking the $L^p_vL^{\infty}_{T}L^{\infty}_{x}$ norm, by $dudu=du'dv'$ and the definition of $\tilde{J}(v)$ \eqref{tildeJ}, we have
\begin{align}\label{J1}
\dis\left\| J_1 \right\|_{L^p_vL^{\infty}_{T}L^{\infty}_{x}}&\leq Cm^{\ga+\frac{3}{p'}}\|\tilde{J}\|_{L^p_v}\notag\\
&\leq Cm^{\ga+\frac{3}{p'}}
\left(  \int_{\R^3}\int_{\R^3}e^{-\frac{|u'|^2}{4}}\|f(v')\|_{L^{\infty}_{T}L^{\infty}_{x}}^p dudv  \right.\notag\\
&\qquad\quad \left.+  \int_{\R^3}\int_{\R^3}e^{-\frac{|v'|^2}{4}}\|f(u')\|_{L^{\infty}_{T}L^{\infty}_{x}}^p dudv  +  \int_{\R^3}\int_{\R^3}e^{-\frac{|v|^2}{4}}\|f(u)\|_{L^{\infty}_{T}L^{\infty}_{x}}^p dudv  \right)^{\frac{1}{p}}\notag\\
&
\leq Cm^{\ga+\frac{3}{p'}}\left\| w_\be f\right\|_{L^p_vL^{\infty}_{T}L^{\infty}_{x}}.
\end{align}
Next we consider $J_3$. It is noted that
\begin{align}\label{J31st}
\left|J_3(t,x,v)\right|&=\left|\int_0^t e^{-\nu(v)(t-s)}(w_\be \Ga)(f,f)(s,x-v(t-s),v)ds\right|\notag\\
&\leq \left|\int_0^t e^{-\nu(v)(t-s)}\nu(v)ds\right|\left\|(w_{\be-\ga} \Ga)(f,f)(v)\right\|_{L^{\infty}_{T_0,T}L^{\infty}_{x}}\notag\\
&\leq \left\|(w_{\be-\ga} \Ga)(f,f)(v)\right\|_{L^{\infty}_{T}L^{\infty}_{x}}.\end{align}
We observe the fact that
\begin{align}\label{Ga-2}
\left\| w_{\be-\ga} \Ga(f,f)\right\|_{L^p_vL^{\infty}_{T}L^{\infty}_{x}}\leq C\left\| w_{\be-\ga} \Ga(f,f)\right\|_{L^p_vL^{\infty}_{T_1}L^{\infty}_{x}}+C\left\| w_{\be-\ga} \Ga(f,f)\right\|_{L^p_vL^{\infty}_{T_1,T}L^{\infty}_{x}}
\end{align}
for some strictly positive constant $C$. Then it follows from \eqref{Ga-1} and \eqref{Ga-2} that
\begin{align}\label{Ga-3}
\left\| w_{\be-\ga}\Ga_-(f,f)\right\|_{L^p_vL^{\infty}_{T}L^{\infty}_{x}}\leq C \left\|f\right\|^{\frac{p-q}{q(p-1)}}_{L^{\infty}_{T_1}L^{\infty}_{x}L^1_v} \left\| w_\be f\right\|^{1+\frac{p(q-1)}{q(p-1)}}_{L^p_vL^{\infty}_{T_1}L^{\infty}_{x}}+C \left\|f\right\|^{\frac{p-q}{q(p-1)}}_{L^{\infty}_{T_1,T}L^{\infty}_{x}L^1_v} \left\| w_\be f\right\|^{1+\frac{p(q-1)}{q(p-1)}}_{L^p_vL^{\infty}_{T_1,T}L^{\infty}_{x}}.
\end{align}
By H\"older's inequality, one gets that 
\begin{align}\label{ObL1Lp1}
&\left\|f(t,x)\right\|_{L^1_v}\leq C\left\|w_\be f(t,x)\right\|_{L^p_v}\leq C \left\| w_\be f\right\|_{L^p_vL^{\infty}_{T_1}L^\infty_x}
\end{align}
for $t\in [0,T_1]$, $\be>3$. Then by \eqref{ObL1Lp1}, \eqref{LE} and the fact that $\frac{p-q}{q(p-1)}+1+\frac{p(q-1)}{q(p-1)}=2$, we have
\begin{align}\label{ObL1Lp}
\left\|f\right\|^{\frac{p-q}{q(p-1)}}_{L^{\infty}_{T_1}L^{\infty}_{x}L^1_v} \left\| w_\be f\right\|^{1+\frac{p(q-1)}{q(p-1)}}_{L^p_vL^{\infty}_{T_1}L^{\infty}_{x}}\leq C \left\|w_\be f\right\|^{\frac{p-q}{q(p-1)}}_{L^p_vL^{\infty}_{T_1}L^{\infty}_{x}} \left\| w_\be f\right\|^{1+\frac{p(q-1)}{q(p-1)}}_{L^p_vL^{\infty}_{T_1}L^{\infty}_{x}}\leq C\left\| w_\be f_0\right\|^2_{L^p_vL^{\infty}_{x}}.
\end{align}
It holds from \eqref{Ga-3}, \eqref{ObL1Lp} that
\begin{align}
\| w_{\be-\ga}\Ga_-(f,f)\|_{L^p_vL^{\infty}_{T}L^{\infty}_{x}}\leq  C\left\| w_\be f_0\right\|_{L^p_vL^\infty_x}^2+C\left\| w_\be f\right\|^{1+\frac{p(q-1)}{q(p-1)}}_{L^p_vL^{\infty}_{T}L^{\infty}_{x}}\left\|f\right\|^{\frac{p-q}{q(p-1)}}_{L^{\infty}_{T_1,T}L^{\infty}_{x}L^1_v} \label{Ga-}.
\end{align}
Using similar arguments as \eqref{Ga-2}, \eqref{Ga-3}, \eqref{ObL1Lp1}, \eqref{ObL1Lp} and the fact that $\frac{1}{8}\left(\frac{1}{q}-\frac{1}{p}\right)+1+\frac{r}{q}+\frac{1}{8}\left(\frac{1}{q}-\frac{1}{p}\right)=2$, one gets the estimate for $\Ga_+$ from \eqref{Ga+1} that
\begin{align}
\| w_{\be-\ga}\Ga_+(f,f)\|_{L^p_vL^{\infty}_{T}L^{\infty}_{x}}\leq  C\left\| w_\be f_0\right\|_{L^p_vL^\infty_x}^2+C\left\| w_\be f\right\|^{\frac{1}{8}\left(\frac{1}{q}-\frac{1}{p}\right)+1+\frac{r}{p}}_{L^p_vL^{\infty}_{T}L^{\infty}_{x}} \left\|f\right\|^{\frac{1}{8}\left(\frac{1}{q}-\frac{1}{p}\right)}_{L^{\infty}_{T_1,T}L^{\infty}_{x}L^1_v} \label{Ga+},
\end{align}
Then it follows from \eqref{J31st}, \eqref{Ga-} and \eqref{Ga+} that
\begin{align}\label{estiJ3}
\left\| J_3 \right\|_{L^p_vL^{\infty}_{T}L^{\infty}_{x}}\leq C&\left\| w_\be f_0\right\|_{L^p_vL^\infty_x}^2+C\left\| w_\be f\right\|^{1+\frac{p(q-1)}{q(p-1)}}_{L^p_vL^{\infty}_{T}L^{\infty}_{x}}\left\|f\right\|^{\frac{p-q}{q(p-1)}}_{L^{\infty}_{T_1,T}L^{\infty}_{x}L^1_v}\notag\\
&+C\left\| w_\be f\right\|^{\frac{1}{8}\left(\frac{1}{q}-\frac{1}{p}\right)+1+\frac{r}{p}}_{L^p_vL^{\infty}_{T}L^{\infty}_{x}} \left\|f\right\|^{\frac{1}{8}\left(\frac{1}{q}-\frac{1}{p}\right)}_{L^{\infty}_{T_1,T}L^{\infty}_{x}L^1_v}.
\end{align}
Obviously it holds that $\left\| J_0 \right\|_{L^p_vL^{\infty}_{T}L^{\infty}_{x}}\leq C\left\| w_\be f_0\right\|_{L^p_vL^{\infty}_{x}}$. Together with \eqref{J1} and \eqref{estiJ3}, we have 
\begin{align}\label{J0J1J3}
\left\| J_0+J_1+J_3 \right\|&_{L^p_vL^{\infty}_{T}L^{\infty}_{x}}
\leq  C\left\| w_\be f_0\right\|_{L^p_vL^{\infty}_{x}}+C\left\| w_\be f_0\right\|_{L^p_vL^\infty_x}^2+Cm^{\ga+\frac{3}{p'}}\left\| w_\be f\right\|_{L^p_vL^{\infty}_{T}L^{\infty}_{x}}\notag\\
+C  &\left\{
 \left\| w_\be f\right\|^{1+\frac{p(q-1)}{q(p-1)}}_{L^p_vL^{\infty}_{T}L^{\infty}_{x}}\left\|f\right\|^{\frac{p-q}{q(p-1)}}_{L^{\infty}_{T_1,T}L^{\infty}_{x}L^1_v}+  \left\| w_\be f\right\|^{\frac{1}{8}\left(\frac{1}{q}-\frac{1}{p}\right)+1+\frac{r}{p}}_{L^p_vL^{\infty}_{T}L^{\infty}_{x}} \left\|f\right\|^{\frac{1}{8}\left(\frac{1}{q}-\frac{1}{p}\right)}_{L^{\infty}_{T_1,T}L^{\infty}_{x}L^1_v} \right\}.
\end{align}
We need to treat $J_2(t,x,v)$ carefully. Let $x_1=x-v(t-s)$. Recall from \eqref{RepKc} that $\left(K^cg\right)(v)=\int_{\R^3}l(v,\eta)g(\eta)d\eta$ and $l_{w_\be}(v,\eta)=l(v,\eta)\frac{w_\be(v)}{w_\be(\eta)}$. Using the mild form \eqref{rmild}, we can rewrite $J_2(t,x,v)$ as
\begin{align*}
\dis J_2(t,x,v)=&\int_0^t e^{-\nu(v)(t-s)}(w_\be K^cf)(s,x_1,v)ds\notag\\
=&\int_0^t e^{-\nu(v)(t-s)}\int_{\R^3}w_\be(v)l(v,\eta)f(s,x_1,\eta)d\eta ds\notag\\
=& \int^t_0 e^{-\nu(v)(t-s)}\int_{\R^3}l_{w_\be}(v,\eta)e^{-\nu(\eta)s}(w_\be f_0)(x_1-\eta s,\eta)d\eta ds\notag\\
&+\int^t_0 e^{-\nu(v)(t-s)}\int_{\R^3}l_{w_\be}(v,\eta)\int^s_0 e^{-\nu(\eta)(s-s_1)}\left(w_\be K^mf\right)(s_1,x_1-\eta(s-s_1),\eta)ds_1d\eta ds\notag\\
&+\int^t_0 e^{-\nu(v)(t-s)}\int_{\R^3}\int_{\R^3}l_{w_\be}(v,\eta)l_{w_\be}(\eta,\xi)\notag\\
&\quad\times \int^s_0e^{-\nu(\eta)(s-s_1)}(w_\be f)(s_1,x_1-\eta(s-s_1),\xi)ds_1d\eta d\xi ds\notag\\
&+\int^t_0 e^{-\nu(v)(t-s)}\int_{\R^3}l_{w_\be}(v,\eta) \notag\\
&\quad\times\int^s_0e^{-\nu(\eta)(s-s_1)}\left(w_\be\Ga(f,f)\right)(s_1,x_1-\eta(s-s_1),\eta)ds_1d\eta ds.
\end{align*}
We take the absolute value of $J_2(t,x,v)$ to obtain
\begin{align*}
|J_2(t,x,v)|\leq J_{20}(t,x,v)+J_{21}(t,x,v)+J_{22}(t,x,v)+J_{23}(t,x,v),
\end{align*}
where
\begin{align*}
J_{20}(t,x,v):=&\int^t_0 e^{-\nu(v)(t-s)}\int_{\R^3}\left|l_{w_\be}(v,\eta)e^{-\nu(\eta)s}(w_\be f_0)(x_1-\eta s,\eta)\right|d\eta ds\notag\\
J_{21}(t,x,v):=&\int^t_0 e^{-\nu(v)(t-s)}\int_{\R^3}\left|l_{w_\be}(v,\eta) \right|\notag\\
&\quad \times\int^s_0\left|e^{-\nu(\eta)(s-s_1)}\left(w_\be K^mf\right)(s_1,x_1-\eta(s-s_1),\eta)\right|ds_1d\eta ds\notag\\
J_{22}(t,x,v):=&\int^t_0 e^{-\nu(v)(t-s)}\int_{\R^3}\int_{\R^3}\left|l_{w_\be}(v,\eta)l_{w_\be}(\eta,\xi)\right|\notag\\
&\quad\times \int^s_0\left|e^{-\nu(\eta)(s-s_1)}(w_\be f)(s_1,x_1-\eta(s-s_1),\xi)\right|ds_1 d\eta d\xi ds\notag\\
J_{23}(t,x,v):=&\int^t_0 e^{-\nu(v)(t-s)}\int_{\R^3}\left|l_{w_\be}(v,\eta) \right|\notag\\
&\quad\times\int^s_0\left|e^{-\nu(\eta)(s-s_1)}\left(w_\be\Ga(f,f)\right)(s_1,x_1-\eta(s-s_1),\eta)\right|ds_1d\eta ds.
\end{align*}
We bound the above four terms $\{J_{2i}\}_{i=0}^3$ one by one. Using the property \eqref{Prol1} and H\"older's inequality, we obtain
\begin{align}
\dis J_{20}(t,x,v&)= \int^t_0 e^{-\nu(v)(t-s)}\int_{\R^3}\left|l_{w_\be}(v,\eta)e^{-\nu(\eta)s}(w_\be f_0)(x_1-\eta s,\eta)\right|d\eta ds\notag\\&\leq \int^t_0 e^{-\nu(v)(t-s)}\int_{\R^3}l_{w_\be}(v,\eta)e^{-\nu(\eta)s}\left|(w_\be f_0)(x_1-\eta s,\eta)\right|d\eta ds\notag\\
&\leq  \int^t_0 e^{-\nu(v)(t-s)}\left(\int_{\R^3}l(v,\eta)\left|\frac{w_\be (v)}{w_\be (\eta)}\right|^{p'}d\eta\right)^{\frac{1}{p'}}
\left(\int_{\R^3}l(v,\eta)\left|(w_\be f_0)(x_1-\eta s,\eta)\right|^pd\eta\right)^{\frac{1}{p}}ds \notag\\
&\leq  C_m\left|\frac{\nu(v)}{(1+|v|)^2}\right|^{\frac{1}{p'}}\int^t_0 e^{-\nu(v)(t-s)}
\left(\int_{\R^3}l(v,\eta)\left|(w_\be f_0)(x_1-\eta s,\eta)\right|^pd\eta\right)^{\frac{1}{p}}ds .\label{estiJ20}
\end{align}
We observe that $\left|(w_\be f_0)(x_1-\eta s,\eta)\right|\leq \left\|(w_\be f_0)(\eta)\right\|_{L^\infty_x}$, then $\int_{\R^3}l(v,\eta)\left\|(w_\be f_0)(\eta)\right\|_{L^\infty_x}^pd\eta$ does not depend on $s$, which together with $\int^t_0 e^{-\nu(v)(t-s)}ds\leq \frac{1}{\nu(v)}$ and \eqref{estiJ20} yield that

\begin{align*}
\dis J_{20}(t,x,v)&\leq  C_m\left|\frac{\nu(v)}{(1+|v|)^2}\right|^{\frac{1}{p'}}\int^t_0 e^{-\nu(v)(t-s)}ds
\left(\int_{\R^3}l(v,\eta)\left\|(w_\be f_0)(\eta)\right\|_{L^\infty_x}^pd\eta\right)^{\frac{1}{p}}\notag\\
&\leq  \frac{C_m}{\nu(v)}\left|\frac{\nu(v)}{(1+|v|)^2}\right|^{\frac{1}{p'}}
\left(\int_{\R^3}l(v,\eta)\left\|(w_\be f_0)(\eta)\right\|_{L^\infty_x}^pd\eta\right)^{\frac{1}{p}}\notag\\
&\leq  C_m
\left(\int_{\R^3}l(v,\eta)\left\|(w_\be f_0)(\eta)\right\|_{L^\infty_x}^pd\eta\right)^{\frac{1}{p}}.
\end{align*}
In the last inequality above, we use the fact that $\frac{1}{\nu(v)}\left|\frac{\nu(v)}{(1+|v|)^2}\right|^{\frac{1}{p'}}=(1+|v|)^{\frac{2-\ga}{p}-2}\leq C$ since $p>  (2-\ga)/2$ from our assumption. Then similar as \eqref{I1}, taking $\|\cdot\|_{L^p_vL^{\infty}_{T}L^\infty_x}$, we have
\begin{align}
\left\| J_{20} \right\|_{L^p_vL^{\infty}_{T}L^\infty_x} \leq C_m\left\| w_\be f_0\right\|_{L^p_vL^{\infty}_{x}}.\label{J20final}
\end{align}
$J_{21}(t,x,v)$ can be estimated in such way. Denote 
 $\eta'=\eta+\left[(\eta_*-\eta)\cdot \omega \right]\omega$, $\eta_*'=\eta_*-\left[(\eta_*-\eta)\cdot \omega \right]\omega,$ and recall from \eqref{tildeJ} that
\begin{align*}
\tilde{J}(\eta)=&\left(  \int_{\R^3}\int_{\S^2}e^{-\frac{|\eta_*'|^2}{4}}\|f(\eta')\|_{L^{\infty}_{T}L^{\infty}_{x}}^p d\omega du  \right)^{\frac{1}{p}}+\left(  \int_{\R^3}\int_{\S^2}e^{-\frac{|\eta'|^2}{4}}\|f(\eta_*')\|_{L^{\infty}_{T}L^{\infty}_{x}}^p d\omega du  \right)^{\frac{1}{p}}\notag\\
&+ \left(  \int_{\R^3}\int_{\S^2}e^{-\frac{|\eta|^2}{4}}\|f(\eta_*)\|_{L^{\infty}_{T}L^{\infty}_{x}}^p d\omega du  \right)^{\frac{1}{p}}.
\end{align*}
By our assumption $p>3/(3+\ga)$, using Lemma \eqref{ProKm}, we obtain the following pointwise bound of $J_{21}(t,x,v)$,

\begin{align*}
\dis J_{21}(t,x,v)&\leq \int^t_0 e^{-\nu(v)(t-s)}\int_{\R^3}\left|l_{w_\be}(v,\eta) \right|\int^s_0e^{-\nu(\eta)(s-s_1)}\left|\left(w_\be K^mf\right)(s_1,x_1-\eta(s-s_1),\eta)\right|ds_1d\eta ds\notag\\
&\leq Cm^{\ga+\frac{3}{p'}}
 \int^t_0 e^{-\nu(v)(t-s)}\int_{\R^3}\left|l_{w_\be}(v,\eta) \right|e^{-\frac{|\eta|^2}{10}}\int^s_0e^{-\nu(\eta)(s-s_1)}ds_1\tilde{J}(\eta)d\eta ds\notag\\
&\leq Cm^{\ga+\frac{3}{p'}}
 \int^t_0 e^{-\nu(v)(t-s)} ds\int_{\R^3}\left|l_{w_\be}(v,\eta) \right|\tilde{J}(\eta)d\eta.
\end{align*}
By similar arguments in \eqref{estiJ20}, it holds that
\begin{align}
\dis J_{21}(t,x,v)&\leq Cm^{\ga+\frac{3}{p'}}\frac{1}{\nu(v)}\left(\int_{\R^3}l_{w_\be}(v,\eta)d\eta\right)^\frac{1}{p'}
\left(\int_{\R^3}l_{w_\be}(v,\eta)\left|\tilde{J}(\eta)\right|^pd\eta\right)^\frac{1}{p}\notag\\
&\leq Cm^{\ga+\frac{3}{p'}}\frac{1}{\nu(v)}
\left|\frac{\nu(v)}{(1+|v|)^2}\right|^{\frac{1}{p'}}\left(\int_{\R^3}l_{w_\be}(v,\eta)\left|\tilde{J}(\eta)\right|^pd\eta\right)^\frac{1}{p}\notag\\
&\leq Cm^{\ga+\frac{3}{p'}}
\left(\int_{\R^3}l_{w_\be}(v,\eta)\left|\tilde{J}(\eta)\right|^pd\eta\right)^\frac{1}{p}.\label{J21}
\end{align}
Then recalling $\|\tilde{J}\|_{L^p_v}\leq C \left\| w_\be f\right\|_{L^p_vL^{\infty}_{T}L^\infty_x}$ in \eqref{J1}, it follows from \eqref{J21} that
\begin{align}
\left\| J_{21} \right\|_{L^p_vL^{\infty}_{T}L^\infty_x} \leq  Cm^{\ga+\frac{3}{p'}}\left\| w_\be f\right\|_{L^p_vL^{\infty}_{T}L^\infty_x}\label{J21final}
\end{align}
We turn to $J_{23}$ now, similar as above, 
\begin{align*}
\dis J_{23}(t,x,v)&\leq \int^t_0 e^{-\nu(v)(t-s)}\int_{\R^3}\left|l_{w_\be}(v,\eta) \right|\int^s_0e^{-\nu(\eta)(s-s_1)}\left|\left(w_\be \Ga(f,f)\right)(s_1,x_1-\eta(s-s_1),\eta)\right|ds_1d\eta ds\notag\\
&\leq C\frac{1}{\nu(v)}\left|\frac{\nu(v)}{(1+|v|)^2}\right|^{\frac{1}{p'}}
\left[\int_{\R^3}l_{w_{\be}}(v,\eta)\left\|\left(w_{\be-\ga}\Ga\left(f,f \right)\right)(\eta)\right\|_{L^{\infty}_{T}L^\infty_x}d\eta\right]^\frac{1}{p}\notag\\
&\leq C
\left[\int_{\R^3}l_{w_{\be}}(v,\eta)\left\|\left(w_{\be-\ga}\Ga\left(f,f\right)\right)(\eta)\right\|_{L^{\infty}_{T}L^\infty_x}d\eta\right]^\frac{1}{p}.
\end{align*}
Then by our estimate \eqref{Ga-} and \eqref{Ga+}, taking the $L^p_vL^{\infty}_{T}L^\infty_x$ norm, we obtain
\begin{align}
\dis \left\| J_{23} \right\|_{L^p_vL^{\infty}_{T}L^\infty_x}\leq C&\left\| w_\be f_0\right\|_{L^p_vL^\infty_x}^2+C  \left\| w_\be f\right\|^{1+\frac{p(q-1)}{q(p-1)}}_{L^p_vL^{\infty}_{T}L^{\infty}_{x}}\left\|f\right\|^{\frac{p-q}{q(p-1)}}_{L^{\infty}_{T_1,T}L^{\infty}_{x}L^1_v}\notag\\
&+  C\left\| w_\be f\right\|^{\frac{1}{8}\left(\frac{1}{q}-\frac{1}{p}\right)+1+\frac{r}{p}}_{L^p_vL^{\infty}_{T}L^{\infty}_{x}} \left\|f\right\|^{\frac{1}{8}\left(\frac{1}{q}-\frac{1}{p}\right)}_{L^{\infty}_{T_1,T}L^{\infty}_{x}L^1_v}\label{J23final}
\end{align}
At last we consider $J_{22}(t,x,v)$. Recall
\begin{align*}
J_{22}(t,x,v)&=\int^t_0 e^{-\nu(v)(t-s)}\int_{\R^3}\int_{\R^3}\left|l_{w_\be}(v,\eta)l_{w_\be}(\eta,\xi)\right|\notag\\
&\quad\times \int^s_0\left|e^{-\nu(\eta)(s-s_1)}(w_\be f)(s_1,x_1-\eta(s-s_1),\xi)\right|ds_1 d\eta d\xi ds.
\end{align*}
 We divide $J_{22}(t,x,v)$ into four cases as \cite{DHWY,Guo}.

\medskip
\noindent{\it Case 1. } $|v|\geq N$. A direct calculation shows that
\begin{align*}
\dis J_{22}(t,x,v)&\leq\int^t_0 e^{-\nu(v)(t-s)}\int_{\R^3}\int_{\R^3}\left|l_{w_\be}(v,\eta)l_{w_\be}(\eta,\xi)\right|\notag\\
&\quad\times \int^s_0\left|e^{-\nu(\eta)(s-s_1)}(w_\be f)(s_1,x_1-\eta(s-s_1),\xi)\right|d\eta d\xi ds_1 ds\notag\\
&\leq\int^t_0 e^{-\nu(v)(t-s)}\int_{\R^3}\int_{\R^3}\left|l_{w_\be}(v,\eta)l_{w_\be}(\eta,\xi)\right|\notag\\
&\quad\times \int^s_0e^{-\nu(\eta)(s-s_1)}ds_1 \left\|(w_\be f)(\xi)\right\|_{L^{\infty}_{T}L^\infty_x}d\eta d\xi ds. 
\end{align*}
We first integrate with respect to $s_1$ first, then integrate with respect to $s$.
\begin{align}\label{EstiJ221st}
\dis J_{22}(t,x,v)&\leq \int^t_0 e^{-\nu(v)(t-s)}ds\int_{\R^3}\int_{\R^3}\left|l_{w_\be}(v,\eta)l_{w_\be}(\eta,\xi)\right|\frac{1}{\nu(\eta)}\left\|(w_\be f)(\xi)\right\|_{L^{\infty}_{T}L^\infty_x} d\eta d\xi \notag\\
&\leq \frac{1}{\nu(v)}\int_{\R^3}\int_{\R^3}\left|l_{w_\be}(v,\eta)l_{w_\be}(\eta,\xi)\right|\frac{1}{\nu(\eta)}\left\|(w_\be f)(\xi)\right\|_{L^{\infty}_{T}L^\infty_x} d\eta d\xi.
\end{align}
Recalling $l_{w_{\be p'}}(v,\eta)=l(\eta,\xi)\left|\frac{w_\be(\eta)}{w_\be(\xi)} \right|^{p'}$ from the notation we define in \eqref{notations}, then by Lemma \ref{Kc} and H\"older's inequality, one obtains that
\begin{align}
\dis J_{22}(t,x,v)&\leq\frac{1}{\nu(v)}\left(\int_{\R^3}\int_{\R^3}l(v,\eta)\left|\frac{w_\be(v)}{w_\be(\eta)} \right|^{p'}l(\eta,\xi)\left|\frac{w_\be(\eta)}{w_\be(\xi)} \right|^{p'}\left|\frac{1}{\nu(\eta)}\right|d\xi d\eta \right)^{\frac{1}{p'}}\notag\\
&\quad\times
\left(\int_{\R^3}\int_{\R^3}l(v,\eta)l(\eta,\xi)\left|\frac{1}{\nu(\eta)}\right|\left|(w_\be f)(\xi)\right|^p_\infty d\eta d\xi \right)^{\frac{1}{p}} \notag\\
&\leq\frac{C_m}{\nu(v)}\left(\int_{\R^3}l_{w_{\be p'}}(v,\eta)\frac{\nu(\eta)}{(1+|\eta|)^2} \left|\frac{1}{\nu(\eta)}\right| d\eta \right)^{\frac{1}{p'}}\notag\\
&\quad\times
\left(\int_{\R^3}\int_{\R^3}l(v,\eta)l(\eta,\xi)\left|\frac{1}{\nu(\eta)}\right|\left\|(w_\be f)(\xi)\right\|^p_{L^{\infty}_{T}L^\infty_x} d\eta d\xi \right)^{\frac{1}{p}}.\label{J22v}
\end{align}
Since $p>(2-\ga)/2$ and $|v|\geq N$, then $\frac{2}{p'}+\frac{\ga}{p}>0$ and 
\begin{align*}
&\frac{C_m}{\nu(v)}\left(\int_{\R^3}l_{w_{\be p'}}(v,\eta)\frac{\nu(\eta)}{(1+|\eta|)^2} \left|\frac{1}{\nu(\eta)}\right| d\eta \right)^{\frac{1}{p'}} \notag\\
&\leq \frac{C_m}{\nu(v)}\left(\int_{\R^3}l_{w_{\be p'}}(v,\eta)d\eta \right)^{\frac{1}{p'}} \notag\\
&\leq C_m\frac{1}{\nu(v)}\left|\frac{\nu(v)}{(1+|v|)^2}\right|^\frac{1}{p'} \notag\\
&\leq C_m\frac{1}{(1+|v|)^{\frac{2}{p'}+\frac{\ga}{p}}} \notag\\
&\leq\frac{C_m}{N^{\frac{2}{p'}+\frac{\ga}{p}}}.
\end{align*}
Substituting the above inequality into \eqref{J22v}, it holds that
\begin{align*}
\dis J_{22}(t,x,v)&\leq\frac{C_m}{N^{\frac{2}{p'}+\frac{\ga}{p}}}
\left(\int_{\R^3}\int_{\R^3}l(v,\eta)l(\eta,\xi)\left|\frac{1}{\nu(\eta)}\right|\left\|(w_\be f)(\xi)\right\|^p_{L^{\infty}_{T}L^\infty_x} d\eta d\xi \right)^{\frac{1}{p}},
\end{align*}
which yields that 
\begin{align}
\dis \left\| J_{22} \right\|_{L^p_vL^{\infty}_{T}L^\infty_x}&\leq \frac{C_m}{N^{\frac{2}{p'}+\frac{\ga}{p}}}\left(\int_{\R^3}\int_{\R^3}\int_{\R^3}l(v,\eta)dv\,l(\eta,\xi)\left|\frac{1}{\nu(\eta)}\right|\left\|(w_\be f)(\xi)\right\|^p_{L^{\infty}_{T}L^\infty_x} d\eta d\xi \right)^{\frac{1}{p}}\notag\\
&\leq \frac{C_m}{N^{\frac{2}{p'}+\frac{\ga}{p}}}\left(\int_{\R^3}\int_{\R^3}l(\eta,\xi)\frac{\nu(\eta)}{(1+|\eta|)^2}\left|\frac{1}{\nu(\eta)}\right|\left\|(w_\be f)(\xi)\right\|^p_{L^{\infty}_{T}L^\infty_x} d\eta d\xi \right)^{\frac{1}{p}}\notag\\
&\leq \frac{C_m}{N^{\frac{2}{p'}+\frac{\ga}{p}}}\left(\int_{\R^3}\int_{\R^3}l(\eta,\xi)\left\|(w_\be f)(\xi)\right\|^p_{L^{\infty}_{T}L^\infty_x} d\eta d\xi \right)^{\frac{1}{p}}\notag\\
&\leq \frac{C_m}{N^{\frac{2}{p'}+\frac{\ga}{p}}}\left\| w_\be f\right\|_{L^p_vL^{\infty}_{T}L^\infty_x}
.\label{J22C1}
\end{align}

\medskip
\noindent{\it Case 2. } $|v|\leq N$, $|\eta|\geq 2N$ or $|\eta|\leq2N$, $|\xi|\geq3N$.
Then either $|\eta-v|\geq N$ or $|\xi-\eta|\geq N$, again by \eqref{Prol4} and \eqref{EstiJ221st}, similar as \eqref{J22v}, when $|\eta-v|\geq N$, we have
\begin{align}\label{EstiJ222nd}
\dis J_{22}(t,x,v)&\leq\frac{1}{\nu(v)}\int_{\R^3}\int_{\R^3}\left|l_{w_{\be p'}}(v,\eta)l_{w_{\be p'}}(\eta,\xi)\right|\frac{1}{\nu(\eta)}\left\|(w_\be f)(\xi)\right\|_{L^{\infty}_{T}L^\infty_x} d\eta d\xi \notag\\
&\leq\frac{1}{\nu(v)}\left(\int_{\R^3}\int_{\R^3}l_{w_{\be p'}}(v,\eta)e^{\frac{|v-\eta|^2}{20}}e^{-\frac{|v-\eta|^2}{20}}l_{w_{\be p'}}(\eta,\xi)\left|\frac{1}{\nu(\eta)}\right|d\eta d\xi \right)^{\frac{1}{p'}}\notag\\
&\quad\times
\left(\int_{\R^3}\int_{\R^3}l(v,\eta)l(\eta,\xi)\left|\frac{1}{\nu(\eta)}\right|\left\|(w_\be f)(\xi)\right\|^p_{L^{\infty}_{T}L^\infty_x} d\eta d\xi \right)^{\frac{1}{p}} \notag\\
&\leq C_m e^{-\frac{N^2}{20}}
\left(\int_{\R^3}\int_{\R^3}l(v,\eta)l(\eta,\xi)\left|\frac{1}{\nu(\eta)}\right|\left\|(w_\be f)(\xi)\right\|^p_{L^{\infty}_{T}L^\infty_x} d\eta d\xi \right)^{\frac{1}{p}}.
\end{align}
The case when $|\xi-\eta|\geq N$ can be estimated in the same way. Then we obtain that
\begin{align}
\dis \left\| J_{22} \right\|_{L^p_vL^{\infty}_{T}L^\infty_x}\leq C_m e^{-\frac{N^2}{20}}\left\| w_\be f\right\|_{L^p_vL^{\infty}_{T}L^\infty_x}.\label{J22C2}
\end{align}

\medskip
\noindent{\it Case 3. } $|v|\leq N$, $|\eta|\leq2N$, $|\xi|\leq3N$, $s-s_1\leq\la$.  Since $e^{-\nu(\eta)(s-s_1)}\leq1$ and $\int^t_0 e^{-\nu(v)(t-s)}ds\leq \frac{1}{\nu(v)}$, one has that
\begin{align}
\dis J_{22}(t,x,v)&=\int^t_0 e^{-\nu(v)(t-s)}\int_{\R^3}\int_{\R^3}\left|l_{w_\be}(v,\eta)l_{w_\be}(\eta,\xi)\right|\notag\\
&\quad\times \int^s_{s-\la}\left|e^{-\nu(\eta)(s-s_1)}(w_\be f)(s_1,x_1-\eta(s-s_1),\xi)\right|ds_1 d\eta d\xi ds\notag\\
&\leq\la\frac{1}{\nu(v)}\int_{\R^3}\int_{\R^3}\left|l_{w_\be}(v,\eta)l_{w_\be}(\eta,\xi)\right|\left\|(w_\be f)(\xi)\right\|_{L^{\infty}_{T}L^\infty_x}d\eta d\xi \notag\\
&\leq C_{m,N} \la
\left(\int_{\R^3}\int_{\R^3}l(v,\eta)l(\eta,\xi)\left\|(w_\be f)(\xi)\right\|^p_{L^{\infty}_{T}L^\infty_x} d\eta d\xi \right)^{\frac{1}{p}},\label{J22trick}
\end{align}
which yields
\begin{align}
\dis \left\| J_{22} \right\|_{L^p_vL^{\infty}_{T}L^\infty_x}\leq C_{m,N} \la\left\| w_\be f\right\|_{L^p_vL^{\infty}_{T}L^\infty_x}.\label{J22C3}
\end{align}

\medskip
\noindent{\it Case 4. } $|v|\leq N$, $|\eta|\leq2N$, $|\xi|\leq3N$, $s-s_1\geq\la$. 

Recall from \eqref{Estionl} that
\begin{align*}
\left| l(v,\eta) \right|\leq\frac{C_\ga}{|v-\eta|^{\frac{3-\ga}{2}}}e^{-\frac{|v-\eta|^2}{10}}e^{-\frac{||v|^2-|\eta|^2|^2}{16|v-\eta|^2}}+C|v-\eta|^\ga e^{-\frac{|v|^2}{4}} e^{-\frac{|\eta|^2}{4}}.
\end{align*}
Since $p>3/(3+\ga)$, $p'\ga>-3$ and $\frac{3-\ga}{2}p'<3$, then $\sup_{v\in \R^3}\left|\int_{\R^3}\left| l_{w_\be}(v,\eta) \right|^{p'}d\eta \right|^\frac{1}{p'} < \infty$.

We can approximate $l_{w_\be}$ by a smooth function $l_N$ with compact support such that
\begin{align}\label{app}
\sup_{|v|\leq3N}\left|\int_{|\eta|\leq3N}\left| l_{w_\be}(v,\eta)-l_N(v,\eta) \right|^{p'}d\eta \right|^\frac{1}{p'} \leq \frac{C_m}{N^{10}}.
\end{align}
We can rewrite $l_{w_\be}(v,\eta)l_{w_\be}(\eta,\xi)=\left(l_{w_\be}(v,\eta)-l_{N}(v,\eta)\right)l_{w_\be}(\eta,\xi)+\left(l_{w_\be}(\eta,\xi)-l_{N}(\eta,\xi)\right)l_{N}(v,\eta)+l_{N}(v,\eta)l_{N}(\eta,\xi)$. A direct calculation shows that 
\begin{align*}
\dis J_{22}(t,x,v)&\leq\int^t_0 e^{-\nu(v)(t-s)}\iint_{|\eta|\leq2N,|\xi|\leq3N}\left|l_{w_\be}(v,\eta)l_{w_\be}(\eta,\xi)\right|\notag\\
&\quad\times \int^{s-\la}_0\left|e^{-\nu(\eta)(s-s_1)}(w_\be f)(s_1,x_1-\eta(s-s_1),\xi)\right|ds_1d\eta d\xi  ds.
\end{align*}
Splitting the right-hand side above into three parts, we have
\begin{align}
\dis J_{22}(t,x,v)&\leq\int^t_0 e^{-\nu(v)(t-s)}\int_{\R^3}\int_{\R^3}\left|l_{w_\be}(v,\eta)-l_{N}(v,\eta)\right|\left|l_{w_\be}(\eta,\xi)\right|\notag\\
&\quad\times \int^{s-\la}_0e^{-\nu(\eta)(s-s_1)}ds_1\left\|(w_\be f)(\xi)\right\|_{L^{\infty}_{T}L^\infty_x}d\eta d\xi ds \notag\\
&+\int^t_0 e^{-\nu(v)(t-s)}\int_{\R^3}\int_{\R^3}\left|l_{w_\be}(\eta,\xi)-l_{N}(\eta,\xi)\right|\left|l_{N}(v,\eta)\right|\notag\\
&\quad\times \int^{s-\la}_0e^{-\nu(\eta)(s-s_1)}ds_1\left\|(w_\be f)(\xi)\right\|_{L^{\infty}_{T}L^\infty_x}d\eta d\xi  ds \notag\\
&+\int^t_0 e^{-\nu(v)(t-s)}\iint_{|\eta|\leq2N,|\xi|\leq3N}\left|l_{N}(v,\eta)l_{N}(\eta,\xi)\right|\notag\\
&\quad\times \int^{s-\la}_0e^{-\nu(\eta)(s-s_1)}\left|(w_\be f)(s_1,x_1-\eta(s-s_1),\xi)\right|ds_1d\eta d\xi  ds.\label{J223}
\end{align}
For the first two terms on the right-hand side of \eqref{J223}, we first integrate with respect to $s_1$, then integrate with respect to $s$ to get
\begin{align}
\dis J_{22}(t,x,v)&\leq\frac{1}{\nu(v)}\iint_{|\eta|\leq2N,|\xi|\leq3N}\left|l_{w_\be}(v,\eta)-l_{N}(v,\eta)\right|\left|l_{w_\be}(\eta,\xi)\right|\frac{1}{\nu(\eta)}\left\|(w_\be f)(\xi)\right\|_{L^{\infty}_{T}L^\infty_x} d\eta d\xi \notag\\
&\quad+\frac{1}{\nu(v)}\iint_{|\eta|\leq2N,|\xi|\leq3N}\left|l_{w_\be}(\eta,\xi)-l_{N}(\eta,\xi)\right|\left|l_{N}(v,\eta)\right|\frac{1}{\nu(\eta)}\left\|(w_\be f)(\xi)\right\|_{L^{\infty}_{T}L^\infty_x} d\eta d\xi\notag\\
&\quad+\int^t_0 e^{-\nu(v)(t-s)}\iint_{|\eta|\leq2N,|\xi|\leq3N}\left|l_{N}(v,\eta)l_{N}(\eta,\xi)\right|\notag\\
&\quad\times \int^{s-\la}_0\left|e^{-\nu(\eta)(s-s_1)}(w_\be f)(s_1,x_1-\eta(s-s_1),\xi)\right|ds_1d\eta d\xi  ds\notag\\
&=J_{221}(t,x,v)+J_{222}(t,x,v)+J_{223}(t,x,v).\label{J22ab}
\end{align}
By the approximation \eqref{app} and the fact that 
$$
\frac{1}{\nu(v)}\frac{1}{\nu(\eta)}\leq CN^6
$$
for $|v|\leq N$, $|\eta|\leq2N$, we yield our estimate for the first term on the right-hand side of \eqref{J22ab},
\begin{align}
\dis J_{221}(t,x,v)=&\frac{1}{\nu(v)}\iint_{|\eta|\leq2N,|\xi|\leq3N}\left|l_{w_\be}(v,\eta)-l_{N}(v,\eta)\right|\left|l_{w_\be}(\eta,\xi)\right|\frac{1}{\nu(\eta)}\left\|(w_\be f)(\xi)\right\|_{L^{\infty}_{T}L^\infty_x}d\eta d\xi \notag\\
& \leq CN^6\iint_{|\eta|\leq2N,|\xi|\leq3N}\left|l_{w_\be}(v,\eta)-l_{N}(v,\eta)\right|\left|l_{w_\be}(\eta,\xi)\right|\left\|(w_\be f)(\xi)\right\|_{L^{\infty}_{T}L^\infty_x}d\eta d\xi\notag\\
& \leq CN^6\left(\int_{|\eta|\leq2N}\int_{|\xi|\leq3N}\left|l_{w_\be}(\eta,\xi)\right|^{p'}d\xi \left|l_{w_\be}(v,\eta)-l_{N}(v,\eta)\right|^{p'}d\eta \right)^\frac{1}{p'}\notag\\
&\quad\times \left( \int_{|\xi|\leq3N}\int_{|\eta|\leq2N}\left\|(w_\be f)(\xi)\right\|^p_{L^{\infty}_{T}L^\infty_x} d\eta d\xi\right)^\frac{1}{p} \notag\\
& \leq CN^9\left(\int_{|\eta|\leq2N}\left|l_{w_\be}(v,\eta)-l_{N}(v,\eta)\right|^{p'}d\eta \right)^\frac{1}{p'}\left( \int_{|\xi|\leq3N}\left\|(w_\be f)(\xi)\right\|^p_{L^{\infty}_{T}L^\infty_x}  d\xi\right)^\frac{1}{p} \notag\\
& \leq \frac{C_m}{N}\left\| w_\be f\right\|_{L^p_vL^{\infty}_{T}L^\infty_x}.\label{JJ221}
\end{align}
Similarly we have
\begin{align}
\dis J_{222}(t,x,v)\leq \frac{C_m}{N}\left\| w_\be f\right\|_{L^p_vL^{\infty}_{T}L^\infty_x}.\label{JJ222}
\end{align}
We turn to $J_{223}$ now, denoting $\nu_N=\inf_{|v|\leq 3N}|\nu(v)|>0$, it holds that
\begin{align}
\dis J_{223}&=\int^t_0 e^{-\nu(v)(t-s)}\iint_{|\eta|\leq2N,|\xi|\leq3N}\left|l_{N}(v,\eta)l_{N}(\eta,\xi)\right|\notag\\
&\quad\times \int^{s-\la}_0\left|e^{-\nu(\eta)(s-s_1)}(w_\be f)(s_1,x_1-\eta(s-s_1),\xi)\right|ds_1d\eta d\xi  ds\notag\\
&\leq\int^t_0 e^{-\nu_N(t-s)}\int^{s-\la}_0e^{-\nu_N(s-s_1)}\iint_{|\eta|\leq2N,|\xi|\leq3N}\left|l_{N}(v,\eta)l_{N}(\eta,\xi)\right|\notag\\
&\quad\times \left|(w_\be f)(s_1,x_1-\eta(s-s_1),\xi)\right|ds_1d\eta d\xi ds \notag\\
&\leq C_{m,N}\int^t_0 e^{-\nu_N(t-s)}\int^{s-\la}_0e^{-\nu_N(s-s_1)}\notag\\
&\quad\times \iint_{|\eta|\leq2N,|\xi|\leq3N}\left|(w_\be f)(s_1,x_1-\eta(s-s_1),\xi)\right|ds_1d\eta d\xi ds .\label{J223in}
\end{align}
We are able to control $J_{223}$ by $\left( \la^{-\frac{3}{2}}\sqrt{\CE(F_0)}+ \la^{-3}\CE(F_0)\right)$ in the following way,
\begin{align}
&\iint_{|\eta|\leq2N, |\xi|\leq3N}\left|(w_\be f)(s_1,x_1-\eta(s-s_1),\xi)\right| d\eta d\xi\notag\\
&\leq C_N\iint_{|\eta|\leq2N, |\xi|\leq3N}\left(\frac{\left|F-\mu \right|}{\sqrt{\mu}}\right)(s_1,x_1-\eta(s-s_1),\xi)\chi_{\{|F(s_1,x_1-\eta(s-s_1),\xi)-\mu(\xi)|\leq \mu(\xi)\}} d\eta d\xi\notag\\
&\quad+ C_N\iint_{|\eta|\leq2N, |\xi|\leq3N}\left|(F-\mu) (s_1,x_1-\eta(s-s_1),\xi)\right|\chi_{\{|F(s_1,x_1-\eta(s-s_1),\xi)-\mu(\xi)|\geq \mu(\xi)\}} d\eta d\xi\notag\\
&\leq C_N\frac{1+(s-s_1)^\frac{3}{2}}{(s-s_1)^\frac{3}{2}}\left\{\int_{\Omega}\int_{ |\xi|\leq3N}\left(\frac{\left|F-\mu \right|^2}{\mu}\right)(s_1,y,\xi)\chi_{\{|F(s_1,y,\xi)-\mu(\xi)|\leq \mu(\xi)\}} dy d\xi\right\}^\frac{1}{2}\notag\\
&\quad+ C_N\frac{1+(s-s_1)^3}{(s-s_1)^3}\int_{\Omega}\int_{ |\xi|\leq3N}\left|(F-\mu) (s_1,y,\xi)\right|\chi_{\{|F(s_1,y,\xi)-\mu(\xi)|\geq \mu(\xi)\}} dyd\xi. \label{UseEntropy}
\end{align}
In the last step above we use the transformation $y=x_1-\eta(s-s_1)$. Substituting \eqref{UseEntropy} into \eqref{J223in} and using Lemma \ref{Taylor}, we obtain
\begin{align}
\dis J_{223}&\leq C_{m,N}\left( \la^{-\frac{3}{2}}\sqrt{\CE(F_0)}+ \la^{-3}\CE(F_0)\right)\label{JJ223}
\end{align}
Then by \eqref{JJ221}, \eqref{JJ222}, \eqref{JJ223}, one has
\begin{align}
\dis \left\| J_{22} \right\|_{L^p_vL^{\infty}_{T}L^\infty_x}\leq \frac{C_m}{N}\left\| w_\be f\right\|_{L^p_vL^{\infty}_{T}L^\infty_x}+C_{m,N}\left( \la^{-\frac{3}{2}}\sqrt{\CE(F_0)}+ \la^{-3}\CE(F_0)\right).\label{J22C4}
\end{align}
In summary of all the four cases, by \eqref{J22C1}, \eqref{J22C2}, \eqref{J22C3} and \eqref{J22C4}, we obtain
\begin{align}
\dis \left\| J_{22} \right\|_{L^p_vL^{\infty}_{T}L^\infty_x}\leq\left(C_{m,N}\la+\frac{C_m}{N^{\frac{2}{p'}+\frac{\ga}{p}}} \right)\left\| w_\be f\right\|_{L^p_vL^{\infty}_{T}L^\infty_x}+
C_{m,N}\left( \la^{-\frac{3}{2}}\sqrt{\CE(F_0)}+ \la^{-3}\CE(F_0)\right).\label{J22final}
\end{align}
Combining our estimates on $J_{20}$ \eqref{J20final}, $J_{21}$ \eqref{J21final}, $J_{22}$ \eqref{J22final}, $J_{23}$ \eqref{J23final}, one gets that
\begin{align}\label{J_2}
\left\| J_2 \right\|_{L^p_vL^{\infty}_{T}L^\infty_x}&\leq  C_m \left\| w_\be f_0\right\|_{L^p_vL^{\infty}_{x}}+\left(Cm^{\ga+\frac{3}{p'}}+C_{m,N}\la+\frac{C_m}{N^{\frac{2}{p'}+\frac{\ga}{p}}}+\frac{C_m}{N} \right)\left\| w_\be f\right\|_{L^p_vL^{\infty}_{T}L^\infty_x}\notag\\
&+C  \left\{
 \left\| w_\be f\right\|^{1+\frac{p(q-1)}{q(p-1)}}_{L^p_vL^{\infty}_{T}L^{\infty}_{x}}\left\|f\right\|^{\frac{p-q}{q(p-1)}}_{L^{\infty}_{T_1,T}L^{\infty}_{x}L^1_v}+  \left\| w_\be f\right\|^{\frac{1}{8}\left(\frac{1}{q}-\frac{1}{p}\right)+1+\frac{r}{p}}_{L^p_vL^{\infty}_{T}L^{\infty}_{x}} \left\|f\right\|^{\frac{1}{8}\left(\frac{1}{q}-\frac{1}{p}\right)}_{L^{\infty}_{T_1,T}L^{\infty}_{x}L^1_v} \right\}\notag\\
&+C\left\| w_\be f_0\right\|_{L^p_vL^{\infty}_{x}}^2+
C_{m,N}\left( \la^{-\frac{3}{2}}\sqrt{\CE(F_0)}+ \la^{-3}\CE(F_0)\right).
\end{align}
It follows from \eqref{rrmild}, \eqref{J0J1J3}, \eqref{J_2} that
\begin{align*}
\left\| w_\be f\right\|_{L^p_vL^{\infty}_{T}L^\infty_x}&\leq  C_m \left\| w_\be f_0\right\|_{L^p_vL^{\infty}_{x}}+\left(Cm^{\ga+\frac{3}{p'}}+C_{m,N}\la+\frac{C_m}{N^{\frac{2}{p'}+\frac{\ga}{p}}}+\frac{C_m}{N} \right)\left\| w_\be f\right\|_{L^p_vL^{\infty}_{T}L^\infty_x}\notag\\
&+C  \left\{
 \left\|f\right\|^{\frac{p-q}{q(p-1)}}_{L^{\infty}_{T}L^\infty_xL^1_v} \left\| w_\be f\right\|^{1+\frac{p(q-1)}{q(p-1)}}_{L^p_vL^{\infty}_{T}L^\infty_x}+  \left\|f\right\|^{\frac{1}{8}\left(\frac{1}{q}-\frac{1}{p}\right)}_{L^{\infty}_{T}L^\infty_xL^1_v} \left\| w_\be f\right\|^{\frac{1}{8}\left(\frac{1}{q}-\frac{1}{p}\right)+1+\frac{r}{p}}_{L^p_vL^{\infty}_{T}L^\infty_x} \right\}\notag\\
&+C\left\| w_\be f_0\right\|_{L^p_vL^{\infty}_{x}}^2+
C_{m,N}\left( \la^{-\frac{3}{2}}\sqrt{\CE(F_0)}+ \la^{-3}\CE(F_0)\right).
\end{align*}
Finally \eqref{LemmaEstimate} holds by first choosing small $m$, then choosing small $\la$ and large $N$.
\end{proof}

\subsection{Smallness of $\|f\|_{L^{\infty}_{T_1,T}L^{\infty}_{x}L^1_v}$}We also need the following lemma, which implies that no matter how large $\left\| w_\be f_0\right\|_{L^p_vL^{\infty}_{x}}$ is, we can choose very small $\CE(F_0),\,\|w_\be f_0\|_{L_x^1L^\infty_v}$ such that $\|f\|_{L^{\infty}_{T_1,T}L^{\infty}_{x}L^1_v}$ will be small.

\begin{lemma}\label{L1v}
Let $\ga$, $\be$ and $p$ satiesfy the assumption in Theorem \ref{local}, $3/(3+\ga)< q< p$, and $T_1$ is the constant given in Theorem \ref{local}. Then for any $T>T_1$, it holds that
\begin{align}\label{LemmaL1v}
\int_{\R^3}\left|f(t,x,v)\right|dv\leq& \int_{\R^3}e^{-\nu(v)t}\left|f_0(x-vt,v)\right|dv+\left(Cm^{\ga+\frac{3}{p'}}+C \la+\frac{C_m}{N}\right) \left\| w_\be f\right\|_{L^p_vL^{\infty}_{T}L^{\infty}_{x}}\notag\\
&+C\left( \la+\frac{1}{N^{{\frac{\be}{2}}-3}} \right)\left\| w_\be f\right\|_{L^p_vL^{\infty}_{T}L^{\infty}_{x}}^2\notag\\
&+C_N\left(\la^{-\frac{3}{2}}\sqrt{\CE(F_0)}+ \la^{-3}\CE(F_0) \right)^{\frac{1}{p'}}\left\| w_\be f\right\|_{L^p_vL^{\infty}_{T}L^{\infty}_{x}}^{1+\frac{1}{p}+\frac{2p(q-1)}{q(p-1)}}\notag\\
&+C_N\left( \la^{-\frac{3}{2}}\sqrt{\CE(F_0)}+ \la^{-3}\CE(F_0)  \right)^{\frac{1}{4}\left(\frac{1}{q}-\frac{1}{p} \right)}
\left\| w_\be f\right\|_{L^p_vL^{\infty}_{T}L^{\infty}_{x}}^{1+\frac{r}{p}},
\end{align}
for any $(t,x)\in[T_1,T]\times \Omega$, where $r=p-\frac{p-q}{4q}$.
\end{lemma}
\begin{proof}
Let $(t,x)\in[T_1,T]\times \Omega$. Using \eqref{rmild}, we have
\begin{equation}
\int_{\R^3}\left|f(t,x,v)\right|dv\leq \int_{\R^3}e^{-\nu(v)t}\left|f_0(x-vt,v)\right|dv+G_1(t,x)+G_2(t,x)+G_3(t,x),\label{L1first}
\end{equation}
where
\begin{align}
&G_1(t,x):=\int_0^t\int_{\R^3} e^{-\nu(v)(t-s)}\left|(K^mf)(s,x-v(t-s),v)\right|dvds \notag\\
&G_2(t,x):=\int_0^t\int_{\R^3} e^{-\nu(v)(t-s)}\left|( K^cf)(s,x-v(t-s),v)\right|dvds\notag\\
&G_3(t,x):=\int_0^t\int_{\R^3} e^{-\nu(v)(t-s)}\left| \Ga(f,f)(s,x-v(t-s),v)\right|dvds.\notag
\end{align}
Here $G_1(t,x)$ can be estimated as \eqref{J21}. Indeed, by the arguments in \eqref{J1} and our definition for $\tilde{J}(v)$ in \eqref{tildeJ} and noticing that $\frac{1}{\nu(v)}e^{-\frac{|v|^2}{10}}\leq Ce^{-\frac{|v|^2}{20}}$, one gets that
\begin{align}\label{G1final}
G_1(t,x)&= \int_0^t\int_{\R^3} e^{-\nu(v)(t-s)}\left|( K^mf)(s,x-v(t-s),v)\right|dvds \notag\\
&\leq Cm^{\ga+\frac{3}{p'}}
 \int^t_0\int_{\R^3} e^{-\nu(v)(t-s)} ds\, e^{-\frac{|v|^2}{10}}\tilde{J}(v)dv\notag\\
&\leq Cm^{\ga+\frac{3}{p'}}
\int_{\R^3} e^{-\frac{|v|^2}{20}}\tilde{J}(v)dv\notag\\
&\leq Cm^{\ga+\frac{3}{p'}}
\left(\int_{\R^3} e^{-\frac{|v|^2}{20}p'}dv\right)^{\frac{1}{p'}}\left(\int_{\R^3} \left|\tilde{J}(v)\right|^pdv\right)^{\frac{1}{p}}\notag\\
&\leq
Cm^{\ga+\frac{3}{p'}}\left\| w_\be f\right\|_{L^p_vL^{\infty}_{T}L^{\infty}_{x}}.
\end{align}
Consider $G_2(t,x)$ in four cases like $J_{22}$. Recall

\begin{align*}
G_2(t,x)&=\int_0^t\int_{\R^3} e^{-\nu(v)(t-s)}\left|( K^cf)(s,x-v(t-s),v)\right|dvds\notag\\
&=\int_0^t\int_{\R^3} e^{-\nu(v)(t-s)}\left|\int_{\R^3}l(v,\eta)f(s,x-v(t-s),\eta)d\eta\right|dvds.
\end{align*}

\noindent{\it Case 1. } $t-\la\leq s\leq t$. By similar arguments as in \eqref{J22trick}, we have
\begin{align*}
G_2(t,x)&=\int_{t-\la}^t\int_{\R^3} e^{-\nu(v)(t-s)}\left|( K^cf)(s,x-v(t-s),v)\right|dvds\notag\\
&\leq \la\int_{\R^3}\int_{\R^3}\frac{1}{w_\be(v)}l(v,\eta)\frac{w_\be(v)}{w_\be(\eta)}\left\|\left(w_\be f\right)(\eta)\right\|_{L^{\infty}_{T}L^{\infty}_{x}} d\eta dv\notag\\
&\leq \la
\left(\int_{\R^3}\int_{\R^3}l(v,\eta)\left|\frac{w_\be(v)}{w_\be(\eta)}\right|^{p'} d\eta\,\frac{1}{w_{\be p'}(v)} dv\right)^\frac{1}{p'}
\left(\int_{\R^3}\int_{\R^3}l(v,\eta)dv\left\|\left(w_\be f\right)(\eta)\right\|^p_{L^{\infty}_{T}L^{\infty}_{x}} d\eta\right)^\frac{1}{p}\notag\\
&\leq \la
\left(\int_{\R^3}\frac{1}{w_{\be p'}(v)} dv\right)^\frac{1}{p'}
\left(\int_{\R^3}\left\|\left(w_\be f\right)(\eta)\right\|^p_{L^{\infty}_{T}L^{\infty}_{x}} d\eta\right)^\frac{1}{p}\notag\\
&\leq C\la \left\| w_\be f\right\|_{L^p_vL^{\infty}_{T}L^{\infty}_{x}}
\end{align*}

\noindent{\it Case 2. } $|\eta|\geq N$. Recall our assumption that $\be>36$, $\frac{1}{\nu(v)w_{\be/2}(v)}$ is bounded. We can obtain $\frac{1}{N}$ from our property of $l(v,\eta)$ in \eqref{Prol2}. Taking the $L^{\infty}_{T}L^{\infty}_{x}$ first and integrating with respect to $s$ like \eqref{EstiJ221st}, it holds that
\begin{align*}
G_2(t,x)&=\int_0^t\int_{\R^3} e^{-\nu(v)(t-s)}\left|\int_{|\eta|\geq N}l(v,\eta)f(s,x-v(t-s),\eta)d\eta\right|dvds\notag\\
&\leq \int_{|\eta|\geq N}\int_{\R^3}\frac{1}{\nu(v)w_{\be/2}(v)}l(v,\eta)\frac{w_{\be/2}(v)}{w_{\be/2}(\eta)}dv\left\|(w_{\be/2} f)(\eta)\right\|_{L^{\infty}_{T}L^{\infty}_{x}} d\eta, \notag
\end{align*}
and further one has 
\begin{align*}
G_2(t,x)
&\leq \frac{C}{N}\int_{|\eta|\geq N}\left\|(w_{\be/2} f)(\eta)\right\|_{L^{\infty}_{T}L^{\infty}_{x}} d\eta\notag\\
&\leq \frac{C}{N}\left(\int_{R^3}\frac{1}{(1+|\eta|)^{\be p'}} d\eta\right)^\frac{1}{p'}\left(\int_{|\eta|\geq N}\left\|(w_{\be} f)(\eta)\right\|^p_{L^{\infty}_{T}L^{\infty}_{x}} d\eta\right)^\frac{1}{p}\notag\\
&\leq \frac{C}{N}\left\| w_\be f\right\|_{L^p_vL^{\infty}_{T}L^{\infty}_{x}}.
\end{align*}

\noindent{\it Case 3. } $|\eta|\leq N$, $|v|\geq 2N$. Then $|v-\eta|\geq N$, similar as \eqref{EstiJ222nd}, by $e^{-\frac{N^2}{20}}\leq \frac{C}{N}$ and $\int_0^t e^{-\nu(v)(t-s)}ds\leq \frac{1}{\nu(v)}$, we obtain
\begin{align}
G_2(t,x)&\leq \iint_{|v-\eta|\geq N}\frac{1}{\nu(v)w_{\be/2}(v)}l(v,\eta)\frac{w_{\be/2}(v)}{w_{\be/2}(\eta)}e^{\frac{|v-\eta|^2}{20}}e^{-\frac{N^2}{20}}\left\|(w_{\be/2} f)(\eta)\right\|_{L^{\infty}_{T}L^{\infty}_{x}} d\eta dv\notag\\
&\leq \frac{C_m}{N}\int_{\R^3}\left\|(w_{\be/2} f)(\eta)\right\|_{L^{\infty}_{T}L^{\infty}_{x}} d\eta\notag\\
&\leq \frac{C_m}{N}\left\| w_\be f\right\|_{L^p_vL^{\infty}_{T}L^{\infty}_{x}}.\label{G2C3}
\end{align}

\noindent{\it Case 4. } $|\eta|\leq N$, $|v|\leq 2N$, $0\leq s \leq t-\la$. Approximate $l_{w_{\be}}$ by $l_N$ as \eqref{app}. Using the similar arguments in \eqref{JJ221}, \eqref{J223in} and \eqref{UseEntropy}, one gets that
\begin{align*}
\dis G_2(t,x)&\leq \int_0^{t-\la}\int_{\R^3} \frac{e^{-\nu(v)(t-s)}}{w_{\be}(v)}\left(\int_{\R^3}\left|l_{w_{\be}}(v,\eta)-l_N(v,\eta)\right|\left| f(s,x-v(t-s),\eta)\right|d\eta\right)dvds\notag\\
&\quad+\int_0^{t-\la}\int_{\R^3} \frac{e^{-\nu(v)(t-s)}}{w_{\be}(v)}\left(\int_{\R^3}\left|l_N(v,\eta)\right|\left|(w_{\be} f)(s,x-v(t-s),\eta)\right|d\eta\right)dvds\notag\\
&\leq \frac{C_m}{N}\left\| w_\be f\right\|_{L^p_vL^{\infty}_{T}L^{\infty}_{x}}+C_{m,N}\int_{\{|\eta|\leq N, |v|\leq 2N\}}\left|(w_{\be} f)(s,x-v(t-s),\eta)\right|d\eta dv\notag\\
&\leq \frac{C_m}{N}\left\| w_\be f\right\|_{L^p_vL^{\infty}_{T}L^{\infty}_{x}}+C_{m,N}\left( \la^{-\frac{3}{2}}\sqrt{\CE(F_0)}+ \la^{-3}\CE(F_0)\right).
\end{align*}
Then $G_2(t,x)$ satiesfies
\begin{align}
\dis G_2(t,x)\leq{C_m}\left(  \la+\frac{1}{N} \right)\left\| w_\be f\right\|_{L^p_vL^{\infty}_{T}L^{\infty}_{x}}+C_{m,N}\left( \la^{-\frac{3}{2}}\sqrt{\CE(F_0)}+ \la^{-3}\CE(F_0)\right).\label{G2final}
\end{align}

At last we need to bound $G_3(t,x)$ which can be divided into three parts. Choose $q$ such that $3/(3+\ga)< q< p$, denote $x_1=x-v(t-s)$. Recall
\begin{align}\label{G3First}
\dis G_3(t,x)&=\int_0^t\int_{\R^3} e^{-\nu(v)(t-s)}\left| \Ga(f,f)(s,x-v(t-s),v)\right|dvds\notag\\
&\leq C\int_0^t\int_{\R^3} e^{-\nu(v)(t-s)} \int_{\R^3} \int_{\S^2} |v-u|^\gamma e^{-\frac{|u|^2}{4}}\notag\\
&\quad \times\left(\left|f(t,x_1,u')f(t,x_1,v')\right|+\left|f(t,x_1,u)f(t,x_1,v) \right|\right)d\omega du dvds.
\end{align}

\noindent{\it Case 1. } $t-\la\leq s\leq t$. It is straightforward to see that $$\left|f(t,x_1,u')f(t,x_1,v')\right|+\left|f(t,x_1,u)f(t,x_1,v) \right|\leq \left\|f(u')f(v') \right\|_{L^{\infty}_{T}L^{\infty}_{x}}+\left\|f(u)f(v) \right\|_{L^{\infty}_{T}L^{\infty}_{x}}.
$$
We now have
\begin{align}
\dis G_3(t,x)&\leq
C\int_{t-\la}^t\int_{\R^3} e^{-\nu(v)(t-s)} \int_{\R^3} \int_{\S^2} |v-u|^\gamma e^{-\frac{|u|^2}{4}}\notag\\
&\quad \times \left(\left\|f(u')f(v') \right\|_{L^{\infty}_{T}L^{\infty}_{x}}+\left\|f(u)f(v) \right\|_{L^{\infty}_{T}L^{\infty}_{x}}\right)d\omega du dvds,\notag
\end{align}
so it hold that
\begin{align}
\dis G_3(t,x)
&\leq  C\la\int_{\R^3} \int_{\R^3} \int_{\S^2} |v-u|^\gamma e^{-\frac{|u|^2}{4}}\left(\left\|f(u')f(v') \right\|_{L^{\infty}_{T}L^{\infty}_{x}}+\left\|f(u)f(v) \right\|_{L^{\infty}_{T}L^{\infty}_{x}}\right)d\omega du dv.\label{G3in}
\end{align}
We observe that  $w_\be(v)\leq Cw_\be(u')w_\be(v')$ and
\begin{align*}
\dis &\left\|f(u')f(v') \right\|_{L^{\infty}_{T}L^{\infty}_{x}}+\left\|f(u)f(v) \right\|_{L^{\infty}_{T}L^{\infty}_{x}}\notag\\
&\leq \frac{1}{w_{\be/2}(v)}\left(Cw_{\be/2}(v)\left\|f(u')f(v') \right\|_{L^{\infty}_{T}L^{\infty}_{x}}+\left\|f(u)\left(w_{\be/2}f\right)(v) \right\|_{L^{\infty}_{T}L^{\infty}_{x}}\right)\notag\\
&\leq \frac{1}{w_{\be/2}(v)}\left(Cw_{\be/2}(u')w_{\be/2}(v')\left\|f(u')f(v') \right\|_{L^{\infty}_{T}L^{\infty}_{x}}+\left\|f(u)\left(w_{\be/2}f\right)(v) \right\|_{L^{\infty}_{T}L^{\infty}_{x}}\right)\notag\\
&\leq \frac{C}{w_{\be/2}(v)}\left(\left\|\left(w_{\be/2}f\right)(u')\left(w_{\be/2}f\right)(v') \right\|_{L^{\infty}_{T}L^{\infty}_{x}}+\left\|\left(w_{\be/2}f\right)(u)\left(w_{\be/2}f\right)(v) \right\|_{L^{\infty}_{T}L^{\infty}_{x}}\right).
\end{align*}
Applying the above inequality to \eqref{G3in} and using H\"older's inequality as \eqref{InequlityGa-}, by $dudv=du'dv'$, we yield
\begin{align}
\dis
G_3(t,x)&\leq C\la\int_{\R^3} \int_{\R^3} \int_{\S^2} \frac{|v-u|^\gamma}{w_{\be/2}(v)} e^{-\frac{|u|^2}{4}}\left(\left\|\left(w_{\be/2}f\right)(u')\left(w_{\be/2}f\right)(v') \right\|_{L^{\infty}_{T}L^{\infty}_{x}}\right.\notag\\
&\qquad\qquad\qquad\qquad\qquad\qquad\qquad\quad\left.+\left\|\left(w_{\be/2}f\right)(u)\left(w_{\be/2}f\right)(v) \right\|_{L^{\infty}_{T}L^{\infty}_{x}}\right)d\omega dudv\notag \\
&\leq C\la\left(\int_{\R^3} \int_{\R^3} \int_{\S^2}\left\|\left(w_{\be/2}f\right)(u')\left(w_{\be/2}f\right)(v') \right\|^q_{L^{\infty}_{T}L^{\infty}_{x}} d\omega dudv\right)^\frac{1}{q}\notag\\
&\quad+ C\la\left(\int_{\R^3} \int_{\R^3} \int_{\S^2}\left\|\left(w_{\be/2}f\right)(u)\left(w_{\be/2}f\right)(v) \right\|^q_{L^{\infty}_{T}L^{\infty}_{x}} d\omega du dv\right)^\frac{1}{q}\notag \\
&\leq C\la \left\| w_\be f\right\|^2_{L^p_vL^{\infty}_{T}L^{\infty}_{x}}.\label{G3C1}
\end{align}

\noindent{\it Case 2. } $|u|\geq N$ or $|v|\geq N$. Set $q'=\frac{q}{1-q}$. It follows from similar arguments as \eqref{G2C3} and \eqref{G3C1} that 
\begin{align}
\dis G_3(t,x)&\leq C\iint_{\{|u|\geq N\}\cup\{|v|\geq N\}} \int_{\S^2} \frac{|v-u|^\gamma}{\nu(v)w_{\be/2}(v)} e^{-\frac{|u|^2}{4}}\left(\left\|\left(w_{\be/2}f\right)(u')\left(w_{\be/2}f\right)(v') \right\|_{L^{\infty}_{T}L^{\infty}_{x}}\right.\notag\\
&\qquad\qquad\qquad\qquad\qquad\qquad\qquad\qquad\left.+\left\|\left(w_{\be/2}f\right)(u)\left(w_{\be/2}f\right)(v) \right\|_{L^{\infty}_{T}L^{\infty}_{x}}\right)d\omega du dv\notag \\
&\leq C\left(\iint_{\{|u|\geq N\}\cup\{|v|\geq N\}}  \int_{\S^2}\left\|\left(w_{\be/2}f\right)(u')\left(w_{\be/2}f\right)(v') \right\|^q_{L^{\infty}_{T}L^{\infty}_{x}} d\omega du dv\right)^\frac{1}{q}\notag\\
&\quad\times\left(\iint_{\{|u|\geq N\}\cup\{|v|\geq N\}} \left| \frac{|v-u|^\ga e^{-\frac{|u|^2}{4}}}{\nu(v)w_{\be/2}(v)} \right|^{q'}dudv\right)^{\frac{1}{q'}}.
\label{G3C2in}
\end{align}
Consider $|v|\geq N$ first. We note that $q>3/(3+\ga)$, $\ga q'>-3$, which yields
\begin{align}
\iint_{|v|\geq N} \left| \frac{|v-u|^\ga e^{-\frac{|u|^2}{4}}}{\nu(v)w_{\be/2}(v)} \right|^{q'}dudv\leq
C\int_{|v|\geq N}\frac{1}{(1+|v|)^{\be q'/2}}dv\leq \frac{C}{N^{\frac{\be q'}{2}-3}}.\label{G3C2v}
\end{align}
Then we turn to $|u|\geq N$. Since $e^{-\frac{|u|^2q'}{4}}$ can be controlled by $\frac{1}{(1+|u|)^\al}$ for any $\al>0$,
\begin{align}
\iint_{|u|\geq N} \left| \frac{|v-u|^\ga e^{-\frac{|u|^2}{4}}}{\nu(v)w_{\be/2}(v)} \right|^{q'}dvdu\leq
C\int_{|u|\geq N}(1+|u|)^{\ga q'}e^{-\frac{|u|^2q'}{4}}du\leq  \frac{C}{N^{\frac{\be q'}{2}-3}}.\label{G3C2u}
\end{align}
Hence after taking $L^p_vL^{\infty}_{T}L^{\infty}_{x}$ norm, by \eqref{G3C2in}, \eqref{G3C2v}, \eqref{G3C2u} and the assumption that $\be>36>-2\ga$, we have
\begin{align}\label{G3C2final}
\dis G_3(t,x)\leq \frac{C}{N^{\frac{\be}{2}-\frac{3}{q'}}}\left\| w_\be f\right\|^2_{L^p_vL^{\infty}_{T}L^{\infty}_{x}}\leq \frac{C}{N^{\frac{\be}{2}-3}}\left\| w_\be f\right\|^2_{L^p_vL^{\infty}_{T}L^{\infty}_{x}}.
\end{align}

\noindent{\it Case 3. } $|u|\leq N$ and $|v|\leq N$, $0\leq s\leq t-\la$, $x_1=x-v(t-s)$. Our first estimate \eqref{G3First} for $G_3(t,x)$ shows that
\begin{align*}
\dis G_3(t,x)&\leq \int_0^{t-\la}\int_{|v|\leq N} e^{-\nu(v)(t-s)}\left(
\int_{|u|\leq N} \int_{\S^2} |v-u|^\gamma e^{-\frac{|u|^2}{4}}\left|f(s,x_1,u)f(s,x_1,v)\right|d\omega du \right)
dvds\notag\\
&\ \, + \int_0^{t-\la}\int_{|v|\leq N} e^{-\nu(v)(t-s)}\left(
\int_{|u|\leq N} \int_{\S^2} |v-u|^\gamma e^{-\frac{|u|^2}{4}}\left|f(s,x_1,u')f(s,x_1,v')\right|d\omega du \right)
dvds\notag\\
&=G_{31}(t,x)+G_{32}(t,x).
\end{align*}
We focus on $G_{31}(t,x)$ first, denote $\nu_N=\inf_{|v|\leq 3N}|\nu(v)|>0$. It follows from the similar arguments in \eqref{InequlityGa-} and \eqref{interpo} that
\begin{align}
\dis G_{31}(t,x)&\leq \int_0^{t-\la}\int_{|v|\leq N} e^{-\nu_N(t-s)}\left(
\int_{|u|\leq N} \int_{\S^2} |v-u|^\gamma e^{-\frac{|u|^2}{4}}\left|f(s,x_1,u)f(s,x_1,v)\right|d\omega du\right)
dvds\notag\\
&\leq C\int_0^{t-\la} e^{-\nu_N(t-s)}\left(\iint_{\{|u|\leq N, |v|\leq N\}}|f(s,x_1,u) |^q |f(s,x_1,v) |^q dudv\right)^\frac{1}{q}ds\notag\\
&\leq C \left\| w_\be f\right\|^{\frac{2p(q-1)}{q(p-1)}}_{L^p_vL^{\infty}_{T}L^{\infty}_{x}} \int_0^{t-\la} e^{-\nu_N(t-s)} \iint_{\{|u|\leq N, |v|\leq N\}}|f(s,x_1,u) | |f(s,x_1,v) | dudvds.\label{G31first}
\end{align}
Also by \eqref{UseEntropy} and Lemma \ref{Taylor}, using H\"older's inequality repeatedly, we obtain
\begin{align*}
 &\iint_{\{|u|\leq N, |v|\leq N\}}|f(s,x-v(t-s),u) | |f(s,x-v(t-s),v) | dudv\notag\\
&\leq\left( \iint_{\{|u|\leq N, |v|\leq N\}}|f(s,x-v(t-s),u) |  dudv\right)^{\frac{1}{p'}}\notag\\
&\qquad\times\left( \iint_{\{|u|\leq N, |v|\leq N\}}\|f(u) \|_{L^{\infty}_{T}L^{\infty}_{x}}\|f(v) \|^p_{L^{\infty}_{T}L^{\infty}_{x}}dvdu\right)^{\frac{1}{p}}\notag\\
&\leq C_N\left\| w_\be f\right\|_{L^p_vL^{\infty}_{T}L^{\infty}_{x}}
\left( \int_{\{|u|\leq N\}}\|f(u) \|_{L^{\infty}_{T}L^{\infty}_{x}}du\right)^{\frac{1}{p}}
\left( \la^{-\frac{3}{2}}\sqrt{\CE(F_0)}+ \la^{-3}\CE(F_0)\right)^\frac{1}{p'}\notag\\
&\leq C_N\left\| w_\be f\right\|^{1+\frac{1}{p}}_{L^p_vL^{\infty}_{T}L^{\infty}_{x}}\left( \la^{-\frac{3}{2}}\sqrt{\CE(F_0)}+ \la^{-3}\CE(F_0)\right)^\frac{1}{p'}.
\end{align*}
Thus, after subtituting the above inequality into \eqref{G31first} and integrating with respect to $s$, we have 
\begin{align}\label{G31final}
\dis G_{31}(t,x)\leq C_N\left\| w_\be f\right\|^{1+\frac{1}{p}+\frac{2p(q-1)}{q(p-1)}}_{L^p_vL^{\infty}_{T}L^{\infty}_{x}}\left( \la^{-\frac{3}{2}}\sqrt{\CE(F_0)}+ \la^{-3}\CE(F_0)\right)^\frac{1}{p'}.
\end{align}
Finally we turn to $G_{32}(t,x)$, as how we treat $G_{31}$ in \eqref{G31first},
\begin{align}
G_{32}(t,x)
&\leq C\int_0^{t-\la} e^{-\nu_N(t-s)}\left(\iint_{\{|u|\leq N, |v|\leq N\}}\int_{\S^2}e^{-\frac{|u|^2}{4}}|f(s,x_1,u') |^q |f(s,x_1,v') |^q d\omega dudv\right)^\frac{1}{q}ds.\label{G32first}
\end{align}
Notice that differently from \eqref{G31first}, this time we keep $e^{-\frac{|u|^2}{4}}$ inside the integral. A similar argument as \eqref{Gama+} shows that
\begin{align*}
\dis &\left(\iint_{\{|u|\leq N, |v|\leq N\}}\int_{\S^2}e^{-\frac{|u|^2}{4}}|f(s,x_1,u') |^q |f(s,x_1,v') |^q d\omega du dv\right)^\frac{1}{q}\notag\\
&\leq   \left(\iint_{\{|u|\leq N, |v|\leq N\}}\int_{\S^2}e^{-\frac{|u|^2}{4}}|f(s,x_1,v') |^\frac{1}{4} d\omega du dv\right)^{\frac{1}{q}-\frac{1}{p} }\notag\\
&\qquad\times
\left(\iint_{\{|u|\leq N, |v|\leq N\}}\int_{\S^2}e^{-\frac{|u|^2}{4}}|f(s,x_1,u') |^p |f(s,x_1,v') |^r d\omega du dv\right)^\frac{1}{p}
\notag\\
&\leq C \left\| w_\be f\right\|^{1+\frac{r}{p}}_{L^p_vL^{\infty}_{T}L^{\infty}_{x}} \left(\int_{\{|u|\leq N, |v|\leq N\}}\int_{\S^2}e^{-\frac{|u|^2}{4}}|f(s,x_1,v') |^\frac{1}{4} d\omega du dv\right)^{\frac{1}{q}-\frac{1}{p} }.
\end{align*}
Since we have $v'=v+\left[(u-v)\cdot \omega \right]\omega$, $|v'|\leq 3N$, $x_1=x-v(t-s)$, it holds that
\begin{align*}
\dis &\iint_{\{|u|\leq N, |v|\leq N\}}\int_{\S^2}e^{-\frac{|u|^2}{4}}|f(s,x_1,v') |^\frac{1}{4} d\omega du dv\notag\\
&\leq C_N\iint_{\{|\eta|\leq 3N, |v|\leq N\}}\int_{z_\perp}|f(s,x_1,\eta)|^\frac{1}{4}\frac{1}{|\eta-v|^{2}}e^{-\frac{|z_\perp+\eta|^2}{4}}dz_\perp d\eta dv \notag \\
&\leq C_N\left( \iint_{\{|\eta|\leq 3N, |v|\leq N\}}|f(s,x_1,\eta)|d\eta dv\right)^\frac{1}{4}
\left( \iint_{\{|\eta|\leq 3N, |v|\leq N\}}\frac{1}{|\eta-v|^{\frac{8}{3}}}d\eta dv\right)^\frac{3}{4}
\notag \\
&\leq C_N\left( \la^{-\frac{3}{2}}\sqrt{\CE(F_0)}+ \la^{-3}\CE(F_0)\right)^\frac{1}{4}.
\end{align*}
Together with \eqref{G32first}, we get
\begin{align}\label{G32final}
\dis G_{32}(t,x)\leq C_N\left( \la^{-\frac{3}{2}}\sqrt{\CE(F_0)}+ \la^{-3}\CE(F_0)  \right)^{\frac{1}{4}\left(\frac{1}{q}-\frac{1}{p} \right)}
\left\| w_\be f\right\|_{L^p_vL^{\infty}_{T}L^{\infty}_{x}}^{1+\frac{r}{p}}.
\end{align}
From \eqref{G31final} and \eqref{G32final}, for \noindent{\it Case 3}, we have
\begin{align}\label{G3C3final}
G_3(t,x)\leq &C_N\left(\la^{-\frac{3}{2}}\sqrt{\CE(F_0)}+ \la^{-3}\CE(F_0) \right)^{\frac{1}{p'}}\left\| w_\be f\right\|_{L^p_vL^{\infty}_{T}L^{\infty}_{x}}^{1+\frac{1}{p}+\frac{2p(q-1)}{q(p-1)}}\notag\\
&+C_N\left( \la^{-\frac{3}{2}}\sqrt{\CE(F_0)}+ \la^{-3}\CE(F_0)  \right)^{\frac{1}{4}\left(\frac{1}{q}-\frac{1}{p} \right)}
\left\| w_\be f\right\|_{L^p_vL^{\infty}_{T}L^{\infty}_{x}}^{1+\frac{r}{p}}.
\end{align}
Using \eqref{G3C1}, \eqref{G3C2final}, \eqref{G3C3final} we obtain the estimate for $G_3(t,x)$ that
\begin{align}
G_3(t,x)\leq &C\left( \la+\frac{1}{N^{{\frac{\be}{2}}-3}} \right)\left\| w_\be f\right\|_{L^p_vL^{\infty}_{T}L^{\infty}_{x}}^2\notag\\
&+C_N\left(\la^{-\frac{3}{2}}\sqrt{\CE(F_0)}+ \la^{-3}\CE(F_0) \right)^{\frac{1}{p'}}\left\| w_\be f\right\|_{L^p_vL^{\infty}_{T}L^{\infty}_{x}}^{1+\frac{1}{p}+\frac{2p(q-1)}{q(p-1)}}\notag\\
&+C_N\left( \la^{-\frac{3}{2}}\sqrt{\CE(F_0)}+ \la^{-3}\CE(F_0)  \right)^{\frac{1}{4}\left(\frac{1}{q}-\frac{1}{p} \right)}
\left\| w_\be f\right\|_{L^p_vL^{\infty}_{T}L^{\infty}_{x}}^{1+\frac{r}{p}}.\label{G3final}
\end{align}
According to \eqref{L1first}, \eqref{G1final}, \eqref{G2final}, \eqref{G3final}, the estimate \eqref{LemmaL1v} follows. This completes the proof of Lemma \ref{L1v}.  
\end{proof}

\subsection{Global existence}
With all the discussions above, we can prove Theorem \ref{global} now. Including the assumptions of Theorem \ref{local} and Theorem \ref{global}, we make the $a$ $priori$ assumption
\begin{align*}
\left\| w_\be f\right\|_{L^p_vL^{\infty}_{T}L^{\infty}_{x}}\leq 2A=2C_2\left(M^2+\sqrt{\CE(F_0)}+\CE(F_0)\right),
\end{align*}
where $M>1$, $\left\| w_\be f_0\right\|_{L^p_vL^{\infty}_{x}}<M$ and $C_2$ is defined in Lemma \ref{Estimate}. Then by Lemma \ref{Estimate}, one gets that
 \begin{align}\label{L1esti1}
\left\| w_\be f\right\|_{L^p_vL^{\infty}_{T}L^{\infty}_{x}}&\leq A+C_2 \left(2A\right)^{1+\frac{p(q-1)}{q(p-1)}}
 \left\|f\right\|_{L^{\infty}_{T_1,T}L^{\infty}_{x}L^1_v}^{\frac{p-q}{q(p-1)}}+C_2 \left(2A\right)^{\frac{1}{8}\left(\frac{1}{q}-\frac{1}{p}\right)+1+\frac{r}{q}}\left\|f\right\|_{L^{\infty}_{T_1,T}L^{\infty}_{x}L^1_v}^{\frac{1}{8}\left(\frac{1}{q}-\frac{1}{p}\right)}.
\end{align}
It follows from Lemma \ref{L1v} that
\begin{align*}
\int_{\R^3}\left|f(t,x,v)\right|dv&\leq
\int_{\R^3}e^{-\nu(v)t}\left|f_0(x-vt,v)\right|dv+\left(Cm^{\ga+\frac{3}{p'}}+C \la+\frac{C_m}{N}\right)(2A)\notag\\
&+C\left( \la+\frac{1}{N^{{\frac{\be}{2}}-3}} \right)(2A)^2\notag\\
&+C_N\left(\la^{-\frac{3}{2}}\sqrt{\CE(F_0)}+ \la^{-3}\CE(F_0) \right)^{\frac{1}{p'}}(2A)^{1+\frac{1}{p}+\frac{2p(q-1)}{q(p-1)}}\notag\\
&+C_N\left( \la^{-\frac{3}{2}}\sqrt{\CE(F_0)}+ \la^{-3}\CE(F_0)  \right)^{\frac{1}{4}\left(\frac{1}{q}-\frac{1}{p} \right)}
(2A)^{1+\frac{r}{p}}.
\end{align*}
Also recall from Theorem \ref{local} that $T_1=\frac{1}{6C_1(1+\|w_\be f_0\|_{L^p_vL^{\infty}_x})}>\frac{C}{M}$. We consider the case that $t\geq T_1$.
If $\Omega=\R^3$,
 \begin{align*}
\int_{\R^3}e^{-\nu(v)t}\left|f_0(x-vt,v)\right|dv \leq \frac{1}{T_1^3}\|f_0\|_{L^1_xL^\infty_v}\leq CM^3\|f_0\|_{L^1_xL^\infty_v}.
\end{align*}
If $\Omega=\T^3$, by $\int_{\{|v|\leq M_1\}}\left|f_0(x-vt,v)\right|dv\leq C\frac{(1+M_1t)^3}{t^3}\int_\Omega\|f_0(y)\|_{L^\infty_v}dy$, we obtain
 \begin{align*}
\dis \int_{\R^3}e^{-\nu(v)t}\left|f_0(x-vt,v)\right|dv &\leq  \int_{\{|v|\geq M_1\}}\left|f_0(x-vt,v)\right|dv+\int_{\{|v|\leq M_1\}}\left|f_0(x-vt,v)\right|dv\notag\\
&\leq \int_{\{|v|\geq M_1\}}\left|f_0(x-vt,v)\right|dv+C\left\{ M_1^3\|f_0\|_{L^1_xL^\infty_v} +M^3\|f_0\|_{L^1_xL^\infty_v}\right\} \notag\\
&\leq  M_1^{\frac{3}{p'}-\be} \left\| w_\be f_0\right\|_{L^p_vL^{\infty}_{x}}+CM_1^3\|f_0\|_{L^1_xL^\infty_v} +CM^3\|f_0\|_{L^1_xL^\infty_v}.
\end{align*}
By choosing $M_1=\left( \frac{\left\| w_\be f_0\right\|_{L^p_vL^{\infty}_{x}}}{ \|f_0\|_{L^1_xL^\infty_v}} \right)^\frac{1}{3+\be-\frac{3}{p'}}$, we have

 \begin{align*}
\dis\int_{\R^3}e^{-\nu(v)t}\left|f_0(x-vt,v)\right|dv &\leq  C\left\| w_\be f_0\right\|^{\frac{3}{3+\be-\frac{3}{p'}}}_{L^p_vL^{\infty}_{x}} \|f_0\|^{1-\frac{3}{3+\be-\frac{3}{p'}}}_{L^1_xL^\infty_v}+CM^3\|f_0\|_{L^1_xL^\infty_v}\notag\\
&\leq  CM^{\frac{3}{3+\be-\frac{3}{p'}}} \|f_0\|^{1-\frac{3}{3+\be-\frac{3}{p'}}}_{L^1_xL^\infty_v}+CM^3\|f_0\|_{L^1_xL^\infty_v}
\end{align*}
Then we can first choose $m$, $\la$ small, $N$ large, and then let $\max\{\CE(F_0),\,\|f_0\|_{L^1_xL^\infty_v}\}\leq \eps$ for some $\eps$ which depends on $\be$, $\ga$, $M$ such that 
 \begin{align}\label{L1esti2}
2C_2 \left(2A\right)^{\frac{p(q-1)}{q(p-1)}}
 \left\|f\right\|_{L^{\infty}_{T_1,T}L^{\infty}_{x}L^1_v}^{\frac{p-q}{q(p-1)}}+2C_2 \left(2A\right)^{\frac{1}{8}\left(\frac{1}{q}-\frac{1}{p}\right)+\frac{r}{q}}\left\|f\right\|_{L^{\infty}_{T_1,T}L^{\infty}_{x}L^1_v}^{\frac{1}{8}\left(\frac{1}{q}-\frac{1}{p}\right)} \leq \frac{1}{2}.
\end{align}
Using \eqref{L1esti1},\eqref{L1esti2}, we directly obtain that 
\begin{align*}
\left\| w_\be f\right\|_{L^p_vL^{\infty}_{T}L^\infty_x}&\leq \frac{3}{2}A.
\end{align*}
We have closed the $a$ $priori$ assumption. Naturally the estimate \eqref{GE} holds. Hence, the proof of Theorem \ref{global} is finished.

\medskip
\noindent {\bf Acknowledgments:}\,
This work is supported by the Hong Kong PhD Fellowship Scheme. The author would like to thank Professor Renjun Duan for his support and encouragement during the PhD studies in CUHK.

\end{document}